%% file: main.tex
\documentclass[11pt]{article}
\input{header}
\input{macro}

\title{History-Aware Adaptive High-Order Tensor Regularization}
\author{
{Chang He} \thanks{The authors are listed alphabetically. School of Information Management and Engineering, Shanghai University of Finance and Economics; Department of Industrial and System Engineering, University of Minnesota. \texttt{ischanghe@gmail.com}}
\and
{Bo Jiang} \thanks{School of Information Management and Engineering, Shanghai University of Finance and Economics. \texttt{isyebojiang@gmail.com}}
\and
{Yuntian Jiang} \thanks{School of Information Management and Engineering, Shanghai University of Finance and Economics. \texttt{yuntianjiang07@gmail.com}}
\and
{Chuwen Zhang} \thanks{Booth School of Business, University of Chicago. \texttt{chuwen.zhang@chicagobooth.edu}}
\and
{Shuzhong Zhang} \thanks{Department of Industrial and System Engineering, University of Minnesota. \texttt{zhangs@umn.edu}}
}

\begin{document}
\maketitle

\begin{abstract}
    In this paper, we develop a new adaptive regularization method for minimizing a composite function, which is the sum of a $p$th-order ($p \ge 1$) Lipschitz continuous function and a simple, convex, and possibly nonsmooth function. We use a history of local Lipschitz estimates to adaptively select the current regularization parameter, an approach we shall term the {\it history-aware adaptive regularization method}. We explore how the selection of an appropriate volume of historical information affects both the theoretical and practical performance. By using all the historical information, our method matches the complexity guarantees of the standard $p$th-order tensor methods that require a known Lipschitz constant, for both convex and nonconvex objectives. In the nonconvex case, the number of iterations required to find an $(\epsilon_g,\epsilon_H)$-approximate second-order stationary point is bounded by $\cO(\max\{\epsilon_g^{-(p+1)/p}, \epsilon_H^{-(p+1)/(p-1)}\})$. For convex functions, we establish an $\cO(\epsilon^{-1/p})$ iteration complexity for finding an $\epsilon$-approximate optimal point and further propose an accelerated variant attaining an iteration complexity of $\cO(\epsilon^{-1/(p+1)})$. For practical consideration, we propose several variants of this method with only part of historical information. We introduce cyclic and sliding-window strategies for choosing historical Lipschitz estimates, which mitigate the limitation of overly conservative updates. As long as a rough upper bound of the Lipschitz constant is known, these two variants achieve the same iteration complexity guarantees in terms of the input accuracy as the method using full historical information. Finally, extensive numerical experiments are conducted to demonstrate the effectiveness of our adaptive approach.
\end{abstract}

\section{Introduction}
\subsection{Motivation}
Adaptively selecting parameters in optimization methods, such as the stepsize in first-order methods, is a central challenge in both theoretical analysis and practical implementation. Typically, there are three common strategies: line-search procedures, ratio test in trust-region methods, and adaptive regularization. The line-search procedure, which can be dated back to \citet{armijo1966minimization}, was designed for gradient methods. It introduces an inner loop to find a suitable stepsize at each iteration. Beyond first-order methods, 
% pure Newton's method is known to lack a global convergence guarantee and can diverge even for a univariate strongly convex functions \citep{polyak1987introduction}. Consequently, 
trust-region methods \citep{conn2000trust} operate on a local quadratic approximation of the objective function within a ball constraint. By introducing a ratio test, the trust-region method adaptively adjusts the radius to guarantee global convergence. These two strategies are mostly designed for specific derivative information, namely, first-order and second-order derivative information. The line-search procedure relies on an explicit update with a closed form, which is generally available in first-order methods, while the trust-region method is based on a quadratic approximation. Unlike these two strategies, the third one 
that works for updates using arbitrary-order derivative information is called \textit{adaptive regularization}, which is the focus of this work.

The regularization technique was initially proposed for Newton's method \citep{griewank1981modification} and is well developed in \cite{conn2000trust}. \cite{nesterov2006cubic} theoretically established the first global complexity guarantees for the cubic regularization of Newton’s method, and \cite{weiser2007affine} then designed an affine-invariant version with encouraging numerical performance. Subsequently, this regularization was extended to a more general $p$th-order tensor update \citep{birgin2017worst,nesterov2021implementable}:
\begin{equation}\label{eq:pth update}
\begin{aligned}
      d^k &= \argmin_{d \in \reals^n} f(x^k) + \sum_{\ell=1}^p \frac{1}{\ell !} \D^\ell f(x^k)[d]^{\otimes \ell} + \frac{\sigma_k}{(p+1)!}\|d\|^{p+1}, \\
      x^{k+1} &= x^k + d^k, \ k = 0, 1, 2, \ldots
\end{aligned}
\end{equation}
where the search direction is generated by minimizing a regularized $p$th-order Taylor polynomial. To ensure the global convergence of the above scheme, one must carefully select the regularization parameter $\sigma_k$ at each iteration $k$, which should be close to the true $p$th-order Lipschitz constant. Therefore, the technique of adaptive regularization was developed to estimate the unknown Lipschitz constant automatically. Generally, there exists two main approaches. \cite{nesterov2006cubic} adopted the idea of the line-search procedure for first-order methods and proposed a line-search adaptive regularization\footnote{We clarify that this is distinct from the line-search procedure commonly used in first-order methods. We adopt the name ``line-search adaptive regularization" because the authors of the original work refer to their adaptive parameter selection scheme as a ``line-search strategy". (see Section 5.2 of \cite{nesterov2006cubic})} (\texttt{LS-AR}) strategy for the cubic regularization case ($p=2$). At each iteration, \texttt{LS-AR} uses an inner loop to repeatedly increase the value of $\sigma_k$ by a factor of 2 until a sufficient decrease condition is met. This strategy was later extended to general $p$th-order tensor updates and to functions with H\"{o}lder continuous derivatives. Another approach is
the adaptive regularization with cubics (\texttt{ARC}) developed by \citep{grapiglia2017regularized,grapiglia2020tensor,grapiglia2022tensor}.
Instead of an additional inner loop, \texttt{ARC} evaluates the search direction at each iteration using a ratio of the actual decrease in the objective function to the decrease predicted by the regularized Taylor polynomial. The regularization parameter $\sigma_k$ and iterate $x^k$ are then updated both based on this ratio. This framework was later generalized to adaptive regularization with $p$th-order models (\texttt{ARP}) \citep{birgin2017worst,cartis2018worst,cartis2022evaluation}. 

These two adaptive regularization schemes are ``Markovian", as their parameter selection depends only on information at the current iterate. \texttt{LS-AR} introduces an inner loop at the current iterate until a proper estimate is found, and the ratio test of \texttt{ARP} is also defined in terms of the current iterate. However, properly using the \textit{history information} can also be beneficial. One evident example is Nesterov's accelerated gradient method \citep{nesterov1983method}, which uses the previous iterate to gain acceleration. Besides, a recent trend in line-search-free first-order methods\footnote{A detailed review is presented in Section \ref{subsection:related work}.} \citep{malitsky2020adaptive,li2025simple} also demonstrates that the past information can effectively estimate the current parameter. These works naturally motivate us to explore the role of history information in designing optimization methods. Therefore, in this paper, we propose an approach termed the \textit{History-aware Adaptive Regularization} (\texttt{HAR}) method. This method leverages Lipschitz constant estimates from past iterates to inform the choice of the current regularization parameter. Our results reveal that while using all previous information is theoretically sound, using only part of the historical information by introducing a budget is often more practical. Our main contributions are presented below.

\subsection{Main contributions}
In this work, motivated by the auto-conditioned gradient descent method \citep{lan2024projected,li2025simple,yagishita2025simple}, we develop the history-aware adaptive regularization method (Algorithm \ref{alg:ac pth-order method}) in parallel with \texttt{LS-AR} and \texttt{ARP}. It is a unified adaptive regularization scheme that works for both convex and nonconvex objectives. Theoretically, we establish the iteration complexity of \texttt{HAR} for the composite objective in problem \eqref{eq:main}. When the differentiable part is convex, \texttt{HAR} finds an $\epsilon$-approximate optimal point (see Definition \ref{def:epsilon optimal point}) in $\cO(\epsilon^{-1/p})$ iterations. Furthermore, an accelerated version, \texttt{HAR-A} (Algorithm \ref{alg:acc HA-AR}), is also proposed for this case, achieving an improved $\cO(\epsilon^{-1/(p+1)})$ iteration complexity. For a nonconvex differentiable part, \texttt{HAR} finds an $(\epsilon_g,\epsilon_H)$-approximate second-order stationary point (see Definition \ref{def:epsilon stationary point}) in at most $\cO(\max\{\epsilon_g^{-(p+1)/p}, \epsilon_H^{-(p+1)/(p-1)}\})$ iterations. These results match those of standard $p$th-order tensor methods when the true Lipschitz constant is given.

On the practical side, we propose two variants: \texttt{HAR-C} (Algorithm \ref{alg:cyclic ac pth-order method}) and \texttt{HAR-S} (Algorithm \ref{alg:slide ac pth-order method}). The original \texttt{HAR} utilizes the local estimates from all the past iterates, which is sometimes not efficient in practice. Therefore, these two variants introduce a predefined budget to limit the number of historical local estimates. This strategy remedies potential overly conservative updates in the practical application of \texttt{HAR}. Under the condition that a rough upper bound of the $p$th-order Lipschitz constant is known, we establish convergence guarantees for these variants that match the original \texttt{HAR} in terms of the input accuracy. To validate the adaptivity of our proposed methods, we test the two practical variants, \harc{} and \hars{}, on numerous applications, ranging from convex to nonconvex problems. With a properly selected budget, our numerical performance shows that both methods, especially \hars{}, are highly comparable to the state-of-the-art implementation of \arc{} \citep{dussaultScalableAdaptiveCubic2024}.

\subsection{Related works}\label{subsection:related work}
In this subsection, we briefly review existing works related to our approach. The first part covers general $p$th-order updates (or high-order methods), which is the setting considered in this work. The second part discusses recently popular line-search-free first-order methods, which motivated us to explore how the selection of an appropriate volume of historical information affects both the theoretical and practical performance of our approach.

\paragraph{High-order optimization methods.} Beyond popular first-order methods \citep{beck2017first}, high-order optimization methods have drawn significant attention over the decades. The seminal work of \cite{nesterov2006cubic} provided the first global convergence rate for a second-order method, which was followed by the development of the more practical version \texttt{ARC} \citep{cartis2011adaptive1,cartis2011adaptive2}. These pioneering works motivated a trend in the complexity analysis of various second-order methods, including trust-region type methods \citep{curtis2017trust,curtis2021trust,jiang2023beyond,jiangAcceleratingTrustregionMethods2025}, homogeneous second-order methods \citep{zhang2025homogeneous,he2025homogeneous}, gradient regularized Newton's methods \citep{mishchenko2023regularized,doikov2024gradient,doikov2024super}, and Newton-CG methods \citep{royer2020newton,yao2023inexact,he2025newton}. Extending the analysis to derivatives of order higher than two, \cite{birgin2017worst,cartis2018worst} proposed \texttt{ARP} and analyzed its iteration complexity for nonconvex optimization. However, a major challenge with this high-order approach was that it was unknown how to efficiently solve the subproblem. A breakthrough for the convex case came when \cite{nesterov2021implementable} proposed an implementable third-order tensor method. More recently, third-order subproblem solvers for nonconvex objectives have also been developed \citep{cartis2024efficient,zhu2025global,cartis2025second,cartis2025cubic,zhou2025tight}. In addition, numerous variants have been proposed for achieving optimal complexity in convex optimization \citep{monteiro2013accelerated,gasnikov2019near,kovalev2022first,carmon2022optimal}, for practical considerations \citep{jiang2020unified,huang2024inexact,xu2020newton}, or as stochastic extensions \citep{hanzely2020stochastic,chen2022accelerating,lucchi2023sub}.

\paragraph{Line-search-free first-order methods.} Line-search-free first-order methods have become an active area of research in recent years. These methods employ adaptive stepsize rules that estimate the Lipschitz constant using information from history iterates, avoiding the inner loops required by traditional line-search procedures. This trend was initiated by two main strategies. The first was proposed by \cite{malitsky2020adaptive} for convex objectives. The second is the auto-conditioned stepsize, introduced by \cite{li2025simple} with an accelerated version for convex optimization (which also works for H\"{o}lder smoothness). Both strategies have motivated numerous variants. The Malitsky-Mishchenko stepsize was extended to the proximal setting for convex problems \citep{malitsky2024adaptive,latafat2024adaptive,oikonomidis2024adaptive} and nonconvex proximal objectives \citep{ye2025simple}, as well as to stochastic optimization \citep{aujol2025stochastic}. Similarly, the auto-conditioned stepsize was also shown to work in the proximal and stochastic settings \citep{yagishita2025simple,lan2024projected}. Beyond these two main approaches, other adaptive strategies have also been developed. For example, \cite{zhou2025adabb} proposed a strategy based on the traditional BB stepsize \citep{barzilai1988two}; \cite{gao2024gradient} developed hypergradient stepsize choice inspired by online learning; \cite{suh2025adaptive} designed an adaptive accelerated gradient method for convex objectives with convergence guarantees on both the function value gap and the gradient norm.

\section{Preliminaries}\label{section.preliminary}
In this paper, we study the following composite optimization problem:
\begin{equation}\label{eq:main}
    \min_{x \in \dom(\psi)} \ F(x) = f(x) + \psi(x),
\end{equation}
where function $f(\cdot)$ is $p$th-order continuously differentiable, and $\psi: \reals^n \to \reals \cup \{\infty\}$ is a simple proper closed convex function with $\dom(\psi) \subseteq \reals^n$. We assume that this problem admits at least one solution, which we denote by $x^\star \in \dom(\psi)$. Since $\psi$ and $F$ are possibly nonsmooth, we use $\partial \psi(x)$ and $\partial F(x)$ to denote their respective subgradients at a point $x$. Given a point \(x \in \mathbb{R}^n\) and a nonempty subset \(\mathcal{S} \subseteq \mathbb{R}^n\), we denote by $\dist(x, \mathcal{S}) \coloneq \inf_{y \in \mathcal{S}} \|x - y\|$ the distance between them. Given a symmetric matrix $A \in \reals^{n \times n}$, we use $\lambda_{\min}(A)$ and $\lambda_{\max}(A)$ to denote its smallest and largest eigenvalues, respectively. The identity matrix is denoted by $\I_{n \times n}$. Now we introduce several approximate optimality measures for problem \eqref{eq:main}. 
\begin{definition}\label{def:epsilon optimal point}
    Given an accuracy $\epsilon > 0$, we say point $x$ is an $\epsilon$-approximate optimal point if it satisfies $F(x) - F(x^\star) \le \epsilon$.
\end{definition}

\begin{definition}\label{def:epsilon stationary point}
    Given accuracies $\epsilon_g, \epsilon_H > 0$, we say point $x$ is an $\epsilon_g$-approximate first-order stationary point if it satisfies $\dist(x, \partial F(x)) \le \epsilon_g$. Furthermore, when $\psi \equiv 0$, we say point $x$ is an $\epsilon_H$-approximate second-order stationary point if it additionally satisfies $\lambda_{\min}(\nabla^2 F(x)) \ge -\epsilon_H$. 
\end{definition}

For $p \ge 1$, the $p$th-order directional derivative of function $f$ at point $x$ along directions $h_i \in \reals^n$, $i=1, \dots, p$ is denoted by $\D^p f(x)[h_1, \dots, h_p]$. If all directions $h_1, \dots, h_p$ are the same, we apply a simpler notation $\D^p f(x)[h]^{\otimes p}$, $h \in \reals^n$. In particular, for any $x \in \reals^n$ and $h_1, h_2 \in \reals^n$, we have
\begin{align*}
    \D f(x)[h_1] = \inner{\nabla f(x)}{h_1}, \ \D^2f(x)[h_1, h_2] = \inner{\nabla^2 f(x)h_1}{h_2}.   
\end{align*}
Then the operator norm of $\D^p f(x)$ is defined by
\begin{align*}
    \|\D^p f(x)\| = \max_{\|h\| \le  1} \ \big|\D^p f(x)[h]^{\otimes p}\big|.
\end{align*}
The $p$th-order Taylor polynomial of function $f$ at point $x$ is defined as follows:
\begin{align*}
    \T_p(y;x) = f(x) + \sum_{\ell=1}^p \frac{1}{\ell !} \D^\ell f(x)[y-x]^{\otimes \ell}, \ \forall y \in \reals^n.
\end{align*}
We use $\Omega_p^\sigma(y;x)$ to denote the $p$th-order Taylor polynomial regularized by a $(p+1)$th-order term with parameter $\sigma$:
\begin{align*}
   \Omega_p^\sigma(y;x) = \T_p(y;x) + \frac{\sigma}{(p+1)!}\|y - x\|^{p+1}, \ \forall y \in \reals^n.
\end{align*}
We say the $p$th-order direction derivative of function $f$ is $L_p$-Lipschitz continuous if it satisfies
\begin{equation}\label{eq:Lipchitz continuous}
    \|\D^p f(x) - \D^p f(y)\| \le L_p \|x - y\|, \ \forall x, y \in \reals^n.
\end{equation}
Under the above Lipschitz continuous property, several Lipschitz-type inequalities hold. For any $p \ge 1$, the function value satisfies
\begin{equation}\label{eq:taylor function value inequality}
    \left|f(y) - \T_p(y;x)\right| \le \frac{L_p}{(p+1)!}\|y - x\|^{p+1}, \ \forall x, y \in \reals^n.
\end{equation}
If $p \ge 2$, the gradient satisfies
\begin{equation}\label{eq:taylor gradient inequality}
    \|\nabla f(y) - \nabla \T_p(y;x)\| \le \frac{L_p}{p!}\|y - x\|^{p},
\end{equation}
and the Hessian matrix satisfies
\begin{equation}\label{eq:taylor Hessian inequality}
    \|\nabla^2 f(y) - \nabla^2 \T_p(y;x)\| \le \frac{L_p}{(p-1)!}\|y - x\|^{p-1}.
\end{equation}
The proof of the above inequalities can be found in \cite{birgin2017worst,cartis2018worst}. Finally, we call a differentiable function $d(x)$ on $\mathbb{R}^n$ uniformly convex (see Section 4.2.2 in \cite{nesterov_lectures_2018}) of degree $p\geq 2$ with constant $q>0$ if
\begin{equation}
    \label{eq.uniformly convex of p}
    d(y)\geq d(x)+ \inner{\nabla d(x)}{y-x} +\frac{q}{p}\|y-x\|^p, \ \forall x,y\in\mathbb{R}^n.
\end{equation}

\section{Algorithm Design and Convergence Analysis}\label{section:HA-AR}

In this section, we first provide the algorithm design of \texttt{HAR}. Similar to standard $p$th-order tensor methods, \texttt{HAR} finds a search direction $d^k$ by minimizing a regularized $p$th-order Taylor polynomial at iteration $k$ (Line \ref{line:compute the direction} of Algorithm \ref{alg:ac pth-order method}). Then a local Lipschitz constant estimate $H_k$ is calculated, which uses the ratio of the Taylor remainder to the magnitude of the search direction:
\begin{align*}
    \frac{H_k}{(p+1)!} = \frac{f(x^{k - 1} + d^k) - \T_p(x^{k-1} + d^k; x^{k-1})}
        {\|d^k\|^{p+1}}.
\end{align*}
Motivated by the auto-conditioned stepsize strategy, we introduce an adaptive parameter $M_k$ as the maximum of all historical local Lipschitz estimates. The final regularization parameter is set as $\sigma_k = \alpha M_k$, where $\alpha > 1$ is a predefined parameter. Compared with the gradient descent method with auto-conditioned stepsize, we introduce a monotone step in \texttt{HAR} (Line \ref{line:argmin step} in Algorithm \ref{alg:ac pth-order method}). Since the adaptive parameter $M_k$ may not be a close approximation of the true Lipschitz constant at some iterations, the direction generated in this scenario could unexpectedly increase the objective value. The monotone step acts as a safeguard: if a sufficient decrease is not achieved, a null iteration is taken instead. This mechanism prevents any undesirable increases in the function value, which is necessary for the convergence analysis.

\begin{algorithm}[h]
  \caption{History-Aware Adaptive Regularization Method}
  \label{alg:ac pth-order method}
  \begin{algorithmic}[1]
    \item \textbf{Input:} Initial point $x^0 \in \dom(\psi)$, initial guess $H_0 = M_0 >0$, parameter $\alpha>1$
    \item \textbf{For} $k=1,2,\ldots$ \textbf{do}
    \item \quad Update the adaptive parameter $M_k = \max\{M_{k-1},H_{k-1}\}$; \label{line:update Mk}
    \item \quad Calculate the regularization parameter $\sigma_k = \alpha M_k$;
    \item \quad Compute the search direction \label{line:compute the direction}
    \begin{equation}\label{eq:sublroblem}
         d^k = \argmin_{d \in \dom(\psi)} \Omega_p^{\sigma_k}(x^{k-1} + d; x^{k-1}) + \psi(x^{k-1} + d);  
    \end{equation}
    \item \quad Update the intermediate iterate $x^{k - 0.5} = x^{k-1} + d^k$; \label{line:intermediate iterate}
    \item \quad Update the iterate $x^k = \argmin_{x\in\{x^{k-0.5}, x^{k-1}\}} F(x)$; \label{line:argmin step}
    \item \quad Update the local Lipschitz estimate \label{line:update local estimate}
    \begin{align*}
       \frac{H_k}{(p+1)!} = \frac{f(x^{k - 0.5}) - \T_p(x^{k-1} + d^k; x^{k-1})}
        {\|d^k\|^{p+1}}; 
    \end{align*} 
    \item \textbf{End For}
  \end{algorithmic}
\end{algorithm}

By the local Lipschitz estimate in Line \ref{line:update local estimate}, we see
\begin{equation}\label{eq:function value identity}
    f(x^{k - 0.5}) - \T_p(x^{k-1} + d^k; x^{k-1}) = \frac{H_k}{(p+1)!}\|d^k\|^{p+1}. 
\end{equation}
The optimality condition of the subproblem \eqref{eq:sublroblem} follows
\begin{equation}\label{eq:function value optimality condition}
    \Omega_p^{\sigma_k}(x^{k-1} + d^k; x^{k-1}) + \psi(x^{k-1} + d^k) \le \Omega_p^{\sigma_k}(x^{k-1}; x^{k-1}) + \psi(x^{k-1}) = f(x^{k-1}) + \psi(x^{k-1}).
\end{equation}
By combining the above two equations \eqref{eq:function value optimality condition}, the function value satisfies
\begin{align*}
    f(x^{k - 0.5}) \le f(x^{k-1}) + \psi(x^{k-1}) - \frac{\sigma_k - H_k}{(p+1)!}\|d^k\|^{p+1} - \psi(x^{k-1} + d^k),
\end{align*}
which can be written equivalently as
\begin{equation}\label{eq:F function value decrease}
    F(x^{k - 0.5}) - F(x^{k-1}) \le - \frac{\sigma_k - H_k}{(p+1)!}\|d^k\|^{p+1}.
\end{equation}
The above descent inequality naturally motivates dividing the sequence of iterates into two categories: successful and unsuccessful. We define these two index sets mathematically as follows
\begin{equation}\label{eq:successful and unsuccessful index sets}
    \cS \coloneq \{k \ge 1 \mid (\alpha + 1)M_k \ge 2H_k\}, \ \text{and} \ \cU \coloneq \mathbb{N} \backslash \cS.
\end{equation}
Clearly, for any $k \ge 1$, it holds that $k = |\cS \cap \{1, \ldots, k\}| + |\cU \cap \{1, \ldots, k\}|$. To refer to the elements of these sets, we let $\cS(j)$ denote the index of the $j$th successful iteration (and similarly for $\cU(j)$). An iteration is successful because the function value is guaranteed to decrease sufficiently:
\begin{align*}
    F(x^{k - 0.5}) - F(x^{k-1}) \le - \frac{(\alpha - 1)M_k}{2(p+1)!}\|d^k\|^{p+1}, \ \forall k \in \cS.
\end{align*}
Conversely, for an unsuccessful iteration $k \in \cU$, a sufficient decrease is not guaranteed, though a decrease may still occur. All null iterations induced by the monotone step are therefore contained within the set of unsuccessful iterations $\cU$. A technical result of our analysis is that the number of unsuccessful iterations is finite, which is motivated by \cite{yagishita2025simple}. The intuition is that \texttt{HAR} performs an implicit tuning during each unsuccessful iteration: the local estimate is implicitly increased, which eventually forces a successful one. This finiteness and the monotonicity together guarantee the convergence of \texttt{HAR}. The following lemma also shows that the choice of the parameter $\alpha$ involves a trade-off between reducing the number of unsuccessful iterations and avoiding overly conservative update in practice.
\begin{lemma}\label{lemma:bounded unsuccessful index set}
    Suppose that function $f$ satisfies \eqref{eq:Lipchitz continuous}. The cardinality of the unsuccessful index set $\cU$ in Algorithm \ref{alg:ac pth-order method} satisfies
    \begin{align*}
        |\cU| \le \overline{\cU} \coloneq \left\lceil\log _{\frac{\alpha + 1}{2}}\frac{H_{\max}}{H_0}\right\rceil,
    \end{align*}
    where $H_{\max} = \max \{H_0, L_p\}$.
\end{lemma}
\begin{proof}
   Note that for any unsuccessful iteration $\cU(j)$, we have $(\alpha + 1)M_{\cU(j)} < 2H_{\cU(j)}$, $j \ge 1$. Recall that $M_{\cU(j) + 1} = \max \{H_0, H_1, \ldots, H_{\cU(j)}\}$. It follows
   \begin{align*}
        M_{\cU(j) + 1} \ge H_{\cU(j)} > \frac{\alpha + 1}{2}M_{\cU(j)}, \ \forall j \ge 1.
   \end{align*}
  Now we prove the result by contradiction. Suppose the cardinality of the unsuccessful index set satisfies $|\cU| > \overline{\cU}$. On the one hand, for any iteration in the region $J \in [\overline{\cU}, |\cU|]$, the above inequality implies
  \begin{align*}
      M_{\cU(J) + 1} > \left(\frac{\alpha + 1}{2}\right)^J M_{\cU(1)} \ge \frac{H_{\max}}{H_0} M_{\cU(1)}.
  \end{align*}
  Since $\cU(1) \ge 1$ and the sequence $\{M_k\}$ is nondecreasing, we have $M_{\cU(1)} \ge M_1 = H_0$.  Substituting this into the inequality above yields $ M_{\cU(J) + 1} > H_{\max}$. On the other hand, from the definition of $H_k$ and inequality \eqref{eq:taylor function value inequality}, we have the following bound for all $k \ge 1$:
  \begin{align*}
    \frac{H_k}{(p+1)!} = \frac{f(x^{k - 0.5}) - \T_p(x^{k-1} + d^k; x^{k-1})}{\|d^k\|^{p+1}} \le \frac{L_p}{(p+1)!}.
  \end{align*}
  This implies that $H_k \le L_p$ for all $k \ge 1$. By the definition of $M_k$, we can then conclude that $M_k \le \max\{L_p, H_0\} = H_{\max}$, $k \ge 1$. Therefore, it holds that $M_{\cU(J) + 1} \le H_{\max}$ and $ M_{\cU(J) + 1} > H_{\max}$, which is a contradiction. The proof is completed.
\end{proof}

\paragraph{Complexity analysis for the convex case.} Now we establish the iteration complexity of \texttt{HAR} under the condition that the differentiable part $f$ in problem \eqref{eq:main} is convex. Our analysis begins with a global perspective on the intermediate iterate in Line \ref{line:intermediate iterate} of Algorithm \ref{alg:ac pth-order method}. First note that Equation \eqref{eq:function value identity} implies
\begin{align*}
    F(x^{k - 0.5}) = \ &f(x^{k - 0.5}) + \psi(x^{k - 0.5}) \\
    = \ &\T_p(x^{k-1} + d^k; x^{k-1}) + \psi(x^{k - 0.5}) + \frac{H_k}{(p+1)!}\|d^k\|^{p+1} \\
    = \ &\Omega^{\sigma_k}_p(x^{k-1} + d^k; x^{k-1}) + \psi(x^{k-1} + d^k) + \frac{H_k - \sigma_k}{(p+1)!}\|d^k\|^{p+1}.
\end{align*}
Since the search direction is generated from the subproblem \eqref{eq:sublroblem}, the above equation follows
\begin{align*}
    F(x^{k - 0.5}) + \frac{\sigma_k - H_k}{(p+1)!}\|d^k\|^{p+1} = \ &\min_{d \in \reals^n} \left\{\Omega^{\sigma_k}_p(x^{k-1} + d; x^{k-1}) + \psi(x^{k-1} + d)\right\} \\
    \le \ &\min_{d \in \reals^n} \left\{f(x^{k-1} + d) + \psi(x^{k-1} + d) + \frac{\sigma_k + L_p}{(p+1)!}\|d\|^{p+1}\right\},
\end{align*}
where we use inequality \eqref{eq:taylor function value inequality} in the second step. Recall that for any successful iteration $\cS(j)$, $(\alpha + 1)M_{\cS(j)} \ge 2H_{\cS(j)}$ for any $j \ge 1$. This implies $\alpha M_{\cS(j)} - H_{\cS(j)} \ge (\alpha - 1)M_{\cS(j)} / 2 > 0$. Consequently, we have the following inequality, which is the main tool for establishing the iteration complexity:
\begin{equation}\label{eq:global inequality successful}
   F(x^{\cS(j)}) = F(x^{{\cS(j)} - 0.5}) \le \min_{d \in \reals^n} \left\{ F(x^{\cS(j)-1} + d) + \frac{\sigma_{\cS(j)} + L_p}{(p+1)!}\|d\|^{p+1}\right\}, \ \forall j \ge 1. 
\end{equation}

\begin{theorem}\label{thm:HA-AR convex}
    Suppose the function $f$ is convex and satisfies \eqref{eq:Lipchitz continuous}. Furthermore, assume that the initial sublevel set $\cF(x^0) \coloneq \{x \in \dom(\psi) \mid F(x) \le F(x^0)\}$ is bounded, i.e., $R \coloneq \sup_{x,y \in \cF(x^0)} \|x - y\| < \infty$. Then the sequence $\{x^k\}_{k \ge 0}$ generated by Algorithm~\ref{alg:ac pth-order method} satisfies
    \begin{align*}
        F(x^k) - F(x^\star) \le \frac{2\alpha H_{\max}R^{p+1}}{p!}\left(\frac{p+1}{k - \overline{\cU}}\right)^p, \ \forall k \ge \overline{\cU}.
    \end{align*}
\end{theorem}
\begin{proof}
    First note that $\|x^k - x^\star\| \le R$, $\forall k \ge 0$ since Algorithm \ref{alg:ac pth-order method} is monotone.
    By applying Equation \eqref{eq:global inequality successful} and convexity of $F$, the function value at any successful iteration $\cS(j)$ satisfies
    \begin{equation}\label{eq:function value decrease F tau}
        \begin{aligned}
            F(x^{\cS(j)}) \le \ &\min_{\tau \in [0, 1]} \left\{ F(x^{\cS(j)-1} + \tau (x^\star - x^{\cS(j)-1})) + \frac{\tau^{p+1}(\sigma_{\cS(j)} + L_p)}{(p+1)!}\|x^\star - x^{\cS(j)-1}\|^{p+1} \right\}\\
        \le \ &\min_{\tau \in [0, 1]} \left\{ F(x^{\cS(j)-1}) - \tau \big(F(x^{\cS(j)-1}) -  F(x^\star)\big) + \frac{\tau^{p+1}(\sigma_{\cS(j)} + L_p)R^{p+1}}{(p+1)!}\right\} \\
        \le \ &\min_{\tau \in [0, 1]} \left\{ F(x^{\cS(j)-1}) - \tau \big(F(x^{\cS(j)-1}) -  F(x^\star)\big) + \frac{2\tau^{p+1}\alpha H_{\max} R^{p+1}}{(p+1)!} \right\}. 
        \end{aligned}
    \end{equation}
    The last step holds because these parameters satisfy $\sigma_{\cS(j)} = \alpha M_{\cS(j)} \le \alpha H_{\max}$, $L_p \le H_{\max}$ and $\alpha > 1$. Let
    \begin{align*}
        \varphi_j(\tau) \coloneq F(x^{\cS(j)-1}) - \tau \big(F(x^{\cS(j)-1}) -  F(x^\star)\big) + \frac{2\tau^{p+1}\alpha H_{\max} R^{p+1}}{(p+1)!}.
    \end{align*}
    The minimum of the univariate function $\varphi_j(\cdot)$ is achieved at
    \begin{align*}
        \tau^\star_{j} = \left(\frac{p!\big(F(x^{{\cS(j)} - 1}) - F(x^\star)\big)}{2\alpha H_{\max} R^{p+1}}\right)^{\frac{1}{p}}.  
    \end{align*}
   Choosing $\tau = 1$ in Equation \eqref{eq:function value decrease F tau} implies
    \begin{equation}\label{eq:upper bound of F at xS(j)}
        F(x^{\cS(j)}) \le F(x^\star) + \frac{2\alpha H_{\max} R^{p+1}}{(p+1)!}, \ \forall j \ge 1.
    \end{equation}
    By combining the above inequality with $F(x^{{\cS(j)} - 1}) \le F(x^{\cS(j-1)})$, we obtain
    \begin{align*}
        \tau^\star_j = \left(\frac{p!\big(F(x^{{\cS(j)} - 1}) - F(x^\star)\big)}{2\alpha H_{\max} R^{p+1}}\right)^{\frac{1}{p}} \le \left(\frac{p!\big(F(x^{\cS(j - 1)}) - F(x^\star)\big)}{2\alpha H_{\max} R^{p+1}}\right)^{\frac{1}{p}} \le \left(\frac{1}{p+1}\right)^{\frac{1}{p}} \le 1, \ \forall j \ge 2.
    \end{align*}
    Therefore, we are safe to substitute $\tau = \tau^\star_j$ into inequality \eqref{eq:function value decrease F tau}. For any $j \ge 2$, this yields the following inequality
    \begin{align*}
        F(x^{\cS(j)}) - F(x^\star) \le F(x^{\cS(j)-1}) - F(x^\star) - \frac{p}{p+1}\left(\frac{p!}{2\alpha H_{\max} R^{p+1}}\right)^{\frac{1}{p}}\big(F(x^{\cS(j)-1}) - F(x^\star)\big)^{\frac{p+1}{p}}.
    \end{align*}
    Let $\Delta_j \coloneq F(x^{\cS(j)}) - F(x^\star)$. Now we claim the following recursive inequality holds in terms of $\Delta_j$:
    \begin{equation}\label{eq:recursive w.r.t. Delta_j}
        \Delta_j \le \Delta_{j-1} - \frac{p}{p+1}\left(\frac{p!}{2\alpha H_{\max} R^{p+1}}\right)^{\frac{1}{p}}\Delta_{j-1}^{\frac{p+1}{p}}, \ \forall j \ge 2.
    \end{equation}
    This is sufficient to show the following inequality holds for any $j \ge 2$:
    \begin{equation}\label{eq:relaxed recursive pth inequality}
        \begin{aligned}
            &F(x^{\cS(j)-1}) - F(x^\star) - \frac{p}{p+1}\left(\frac{p!}{2\alpha H_{\max} R^{p+1}}\right)^{\frac{1}{p}}\big(F(x^{\cS(j)-1}) - F(x^\star)\big)^{\frac{p+1}{p}} \\
            \le \ &\Delta_{j-1} - \underbrace{\frac{p}{p+1}\left(\frac{p!}{2\alpha H_{\max} R^{p+1}}\right)^{\frac{1}{p}}}_{\coloneq C}\Delta_{j-1}^{\frac{p+1}{p}}. 
        \end{aligned}
    \end{equation}
    This can be proved by investigating the univariate function $\vartheta(\tau) \coloneq \tau - C \cdot \tau^{(p+1) / p}$. It is easy to verify that $\vartheta(\cdot)$ is nondecreasing in the region $[0, (p/((p+1)C))^p]$. Equation \eqref{eq:upper bound of F at xS(j)} guarantees that 
    \begin{align*}
        \Delta_{j-1} \le \frac{2\alpha H_{\max} R^{p+1}}{(p+1)!} \le \left(\frac{p}{(p+1)C}\right)^p, \ \forall j \ge 2.
    \end{align*}
    Since $F(x^{\cS(j) - 1})  - F(x^\star) \le F(x^{\cS(j-1)}) - F(x^\star)$, $\forall j \ge 2$, the univariate function $\vartheta(\cdot)$ satisfies
    \begin{align*}
        \vartheta(F(x^{\cS(j) - 1}) - F(x^\star)) \le \vartheta(F(x^{\cS(j-1)}) - F(x^\star)) = \vartheta(\Delta_{j-1}), \ \forall j \ge 2.
    \end{align*}
    This is exactly the inequality from Equation \eqref{eq:relaxed recursive pth inequality}. To this end, we derive the recursive inequality in terms of $\Delta_j$ from Equation \eqref{eq:recursive w.r.t. Delta_j}:
    \begin{align*}
        \Delta_{j - 1} - \Delta _j \ge C \cdot \Delta_{j-1}^{\frac{p+1}{p}}, \ \forall j \ge 2.
    \end{align*}
     By a similar argument to the proof of Theorem 2 in \cite{nesterov2021implementable}, we conclude
     \begin{align*}
        \Delta_j \le C^{-p} \left(\frac{1}{C \Delta_1^{1/p}} + \frac{j-1}{p}\right)^{-p} \le C^{-p} \left(\frac{(p+1)^{(p+1)/p}}{p} + \frac{j-1}{p}\right)^{-p}, \ \forall j \ge 1, 
     \end{align*}
     which implies
     \begin{equation}\label{eq:convergence of successful iterates}
         F(x^{\cS(j)}) - F(x^\star) \le \frac{2\alpha H_{\max}R^{p+1}}{p!}\left(\frac{p+1}{j}\right)^p, \ \forall j \ge 1. 
     \end{equation}
    Note that for any iteration count $k$, it holds that $k = |\cS \cap \{1, \ldots, k\}| +  |\cU \cap \{1, \ldots, k\}|$. Let $j(k) \coloneqq |\cS \cap \{1, \ldots, k\}|$. Then for any $k \ge \overline{\cU}$, we see
    \begin{align*}
        j(k) = k - |\cU \cap \{1, \ldots, k\}| \ge k - |\cU| \ge k - \overline{\cU}.
    \end{align*}
    Since the method is monotone, we have $F(x^k) \le F(x^{\cS(j(k))})$, which implies
    \begin{align*}
        F(x^k) - F(x^\star) \le \frac{2\alpha H_{\max}R^{p+1}}{p!}\left(\frac{p+1}{j(k)}\right)^p \le \frac{2\alpha H_{\max}R^{p+1}}{p!}\left(\frac{p+1}{k - \overline{\cU}}\right)^p, \ \forall k \ge \overline{\cU}.
    \end{align*}
    The proof is completed.
\end{proof}

\begin{corollary}\label{corollary:HA-AR convex}
     Suppose the function $f$ is convex and satisfies \eqref{eq:Lipchitz continuous}. Furthermore, assume that the initial sublevel set $\cF(x^0) \coloneq \{x \in \dom(\psi) \mid F(x) \le F(x^0)\}$ is bounded, i.e., $R \coloneq \sup_{x,y \in \cF(x^0)} \|x - y\| < \infty$. Let the sequence $\{x^k\}_{k \ge 0}$ be generated by Algorithm~\ref{alg:ac pth-order method}. Then for any accuracy $\epsilon > 0$, Algorithm \ref{alg:ac pth-order method} requires at most
    \begin{align*}
        K = \overline{\cU} + \mathcal{O}\left(\left(\frac{\alpha H_{\max} R^{p+1}}{\epsilon}\right)^{\frac{1}{p}}\right)
    \end{align*}
    iterations to return an iterate $x^K$ such that $F(x^K) - F(x^\star) \le \epsilon$.
\end{corollary}

\paragraph{Complexity analysis for the nonconvex case.} We now turn to the case where the differentiable part $f$ is nonconvex. Since the objective function $F$ in problem \eqref{eq:main} is nonsmooth, we measure the stationarity of a point $x \in \dom(\psi)$ by the distance $\dist(0, \partial F(x))$. Recalling the function value decrease from \eqref{eq:F function value decrease}, the analysis reduces to relating the magnitude of the search direction $d^k$ and to the stationarity measure $\dist(0, \partial F(x^{k - 0.5}))$.
\begin{lemma}\label{lemma:relate d with gradient Hessian}
   Suppose that function $f$ satisfies \eqref{eq:Lipchitz continuous}. For any successful iteration, it holds that
   \begin{align*}
       \dist(0, \partial F(x^{\cS(j)})) \le \frac{2\alpha H_{\max}}{p!}\|d^{\cS(j)}\|^p, \ \forall j \ge 1.
   \end{align*}
    Furthermore, if the nonsmooth term vanishes, i.e, $\psi \equiv 0$, it also holds that
    \begin{align*}
        -\lambda_{\min}\big(\nabla^2 F(x^{\cS(j)})\big) \le \frac{2\alpha H_{\max}}{(p-1)!}\|d^{\cS(j)}\|^{p-1}, \ \forall j \ge 1.
    \end{align*}
\end{lemma}
\begin{proof}
    Let $j \ge 1$. By definition of successful iterations, the function value decreases at iteration $\cS(j)$ and hence $x^{\cS(j)} = x^{\cS(j) - 1} + d^{\cS(j)}$. By the first-order optimality of subproblem \eqref{eq:sublroblem}, we have
    \begin{align*}
        \nabla \Omega_p^{\sigma_{\cS(j)}}(x^{\cS(j)}; x^{\cS(j)-1}) + \psi^\prime(x^{\cS(j)}) = 0,
    \end{align*}
    where $\psi^\prime(x^{\cS(j)}) \in \partial \psi(x^{\cS(j)})$. Note that
    \begin{align*}
        &\nabla f(x^{\cS(j)}) + \psi^\prime(x^{\cS(j)}) \\
        = \ &\nabla f(x^{\cS(j)}) - \nabla \Omega_p^{\sigma_{\cS(j)}}(x^{\cS(j)}; x^{\cS(j)-1}) + \nabla \Omega_p^{\sigma_{\cS(j)}}(x^{\cS(j)}; x^{\cS(j)-1}) + \psi^\prime(x^{\cS(j)}) \\
        = \ &\nabla f(x^{\cS(j)}) - \nabla \T_p(x^{\cS(j)}; x^{\cS(j)-1}) - \frac{\sigma_{\cS(j)}}{p!}\|d^{\cS(j)}\|^{p-1}d^{\cS(j)}.
    \end{align*}
    Combining the above equation with inequality \eqref{eq:taylor gradient inequality}, we obtain
    \begin{align*}
        &\|\nabla f(x^{\cS(j)}) + \psi^\prime(x^{\cS(j)})\| \\
        \le \ &\|\nabla f(x^{\cS(j)}) - \nabla \T_p(x^{\cS(j)}; x^{\cS(j)-1})\| + \frac{\sigma_{\cS(j)}}{p!}\|d^{\cS(j)}\|^p \\
        \le \ &\frac{L_p + \sigma_{\cS(j)}}{p!}\|d^{\cS(j)}\|^p \\
        \le \ &\frac{2\alpha H_{\max}}{p!}\|d^{\cS(j)}\|^p.
    \end{align*}
    We use inequalities $\sigma_{\cS(j)} = \alpha M_{\cS(j)} \le \alpha H_{\max}$ and $L_p \le H_{\max} \le \alpha H_{\max}$ in the last step. As a result, by applying the relationship $\dist(0, x^{\cS(j)}) \le \|\nabla f(x^{\cS(j)}) + \psi^\prime(x^{\cS(j)})\|$, we have $\dist(0, \partial F(x^{\cS(j)})) \le 2\alpha L_p \|d^{\cS(j)}\|^p / (p!)$. Now we consider the case that $\psi \equiv 0$. First note that
    \begin{align*}
        &\nabla^2 \Omega_p^{\sigma_{\cS(j)}}(x^{\cS(j)}; x^{\cS(j)-1}) \\
        = \ &\nabla^2 \T_p(x^{\cS(j)}; x^{\cS(j)-1}) + \frac{\sigma_{\cS(j)}}{p!}\big((p-1)\|d^{\cS(j)}\|^{p-3}d^{\cS(j)}(d^{\cS(j)})^\top + \|d^{\cS(j)}\|^{p-1}\I_{n \times n}\big).
    \end{align*}
    By the second-order optimality of subproblem \eqref{eq:sublroblem}, we have
    \begin{align*}
        \nabla^2 \Omega_p^{\sigma_{\cS(j)}}(x^{\cS(j)}; x^{\cS(j)-1}) \succeq 0,
    \end{align*}
    which implies
    \begin{align*}
        \nabla^2 \T_p(x^{\cS(j)}; x^{\cS(j)-1}) \succeq -\frac{\sigma_{\cS(j)}}{p!}\big((p-1)\|d^{\cS(j)}\|^{p-3}d^{\cS(j)}(d^{\cS(j)})^\top + \|d^{\cS(j)}\|^{p-1}\I_{n \times n}\big).
    \end{align*}
    By inequality \eqref{eq:taylor Hessian inequality}, it follows
    \begin{align*}
       &\nabla^2 F(x^{\cS(j)}) \\
       \succeq \ &\nabla^2 \T_p(x^{\cS(j)}; x^{\cS(j)-1}) - \frac{L_p}{(p-1)!}\|d^{\cS(j)}\|^{p-1} \I_{n \times n} \\
       \succeq \ &-\frac{\sigma_{\cS(j)}}{p!}\big((p-1)\|d^{\cS(j)}\|^{p-3}d^{\cS(j)}(d^{\cS(j)})^\top + \|d^{\cS(j)}\|^{p-1}\I_{n \times n}\big) - \frac{L_p}{(p-1)!}\|d^{\cS(j)}\|^{p-1} \I_{n \times n} \\
       = \ &-\frac{(p-1)\sigma_{\cS(j)}}{p!}\|d^{\cS(j)}\|^{p-3}d^{\cS(j)}(d^{\cS(j)})^\top - \frac{\sigma_{\cS(j)} + p L_p}{p!}\|d^{\cS(j)}\|^{p-1} \I_{n \times n}.
    \end{align*}
    Therefore, we conclude
    \begin{align*}
        \lambda_{\min}\big(\nabla^2 F(x^{\cS(j)})\big) \ge \ &-\frac{(p-1)\sigma_{\cS(j)}}{p!}\|d^{\cS(j)}\|^{p-3} \lambda_{\max}\big(d^{\cS(j)}(d^{\cS(j)})^\top\big) - \frac{\sigma_{\cS(j)} + p L_p}{p!}\|d^{\cS(j)}\|^{p-1} \\
        = \ &-\frac{(p-1)\sigma_{\cS(j)}}{p!}\|d^{\cS(j)}\|^{p-1} - \frac{\sigma_{\cS(j)} + p L_p}{p!}\|d^{\cS(j)}\|^{p-1} \\
        = \ &-\frac{\sigma_{\cS(j)} + L_p}{(p-1)!}\|d^{\cS(j)}\|^{p-1}.
    \end{align*}
    The proof is completed by using inequalities $\sigma_{\cS(j)} = \alpha M_{\cS(j)} \le \alpha H_{\max}$ and $L_p \le H_{\max} \le \alpha H_{\max}$ again.
\end{proof}

\begin{theorem}\label{thm:HA-AR nonconvex}
   Suppose that function $f$ satisfies \eqref{eq:Lipchitz continuous}. Let the sequence $\{x^k\}_{k \ge 0}$ be generated by Algorithm~\ref{alg:ac pth-order method}. Then for any accuracies $\epsilon_g, \epsilon_H > 0$, the algorithm requires at most
   \begin{align*}
       K = \overline{\cU} + \cO\left(\frac{F(x^0) - F(x^\star)}{(\alpha - 1)H_0} \cdot \max\left\{\left(\frac{\alpha H_{\max}}{ \epsilon_g} \right)^{\frac{p+1}{p}}, \left(\frac{\alpha H_{\max}}{\epsilon_H} \right)^{\frac{p+1}{p-1}}\right\}\right)
   \end{align*}
   iterations to return an iterate $\Tilde{x}$ such that $\dist(0, \partial F(\Tilde{x})) \le \epsilon_g$. Furthermore, if the nonsmooth term vanishes, i.e, $\psi \equiv 0$, the sequence additionally satisfies $\lambda_{\min}(\nabla^2 F(\Tilde{x})) \ge -\epsilon_{H}$.
\end{theorem}
\begin{proof}
    Recall the definition of the successful index set in \eqref{eq:successful and unsuccessful index sets}. Let $j \ge 1$. We have
    \begin{align*}
        F(x^{\cS(j)}) - F(x^{\cS(j)-1}) \le - \frac{(\alpha -1)M_{\cS(j)}}{2(p+1)!}\|d^{\cS(j)}\|^{p+1}.
    \end{align*}
    Summing over $j$ from $1$ to $J$ implies
    \begin{align*}
        \sum_{j=1}^J \frac{(\alpha -1)M_{\cS(j)}}{2(p+1)!}\|d^{\cS(j)}\|^{p+1} \le \sum_{j=1}^J F(x^{\cS(j)-1}) - F(x^{\cS(j)}) \le \sum_{k=1}^{\cS(J)} F(x^{k-1}) - F(x^{k}).
    \end{align*}
    where we use the monotonicity of Algorithm \ref{alg:ac pth-order method}. Since $M_{\cS(j)} \ge H_0$, $\forall j \ge 1$, the above inequality follows
    \begin{align*}
        \sum_{j=1}^J \frac{(\alpha -1)H_0}{2(p+1)!}\|d^{\cS(j)}\|^{p+1} \le F(x^0) - F(x^\star), 
    \end{align*}
    which yields
    \begin{equation}\label{eq:d min nonconve inequality}
        \|d^{\cS(j(\min))}\|^{p+1} \coloneq \min_{j = 1, \ldots, J} \|d^{\cS(j)}\|^{p+1} \le \frac{2(p+1)!(F(x^0) - F(x^\star))}{(\alpha - 1)H_0 J}. 
    \end{equation}
    Clearly, $\cS(j(\min)) \in \{\cS(1), \ldots, \cS(J)\}$. Let $K = \cS(J)$. Note that $K = J + |\cU \cap \{1, \ldots, K\}| \le J + \overline{\cU}$. The above inequality follows
    \begin{align*}
        \|d^{\cS(j(\min))}\|^{p+1} \le \frac{2(p+1)!(F(x^0) - F(x^\star))}{(\alpha - 1)H_0 (K - \overline{\cU})}. 
    \end{align*}
    Now we use the results in Lemma \ref{lemma:relate d with gradient Hessian} to get
    \begin{align*}
       \dist(0, \partial F(x^{\cS(j(\min))})) \le \frac{2\alpha H_{\max}}{p!} \left( \frac{2(p+1)!(F(x^0) - F(x^\star))}{(\alpha - 1)H_0(K - \overline{\cU})} \right)^{\frac{p}{p+1}}.
    \end{align*}
    Therefore, after at most 
    \begin{align*}
        K = \overline{\cU} + \left\lceil \frac{2(p + 1)!(F(x^0) - F(x^\star))}{(\alpha - 1)H_0}\left(\frac{2\alpha H_{\max}}{p! \epsilon_g} \right)^{\frac{p+1}{p}}\right\rceil
    \end{align*}
    iterations, it holds that $\dist(0, \partial F(x^{\cS(j(\min))})) \le \epsilon_g$. In the special case where $\psi \equiv 0$ holds, from Lemma \ref{lemma:relate d with gradient Hessian}, we also have
    \begin{align*}
        -\lambda_{\min}\big(\nabla^2 F(x^{\cS(j(\min))})\big) \le \frac{2\alpha H_{\max}}{(p-1)!} \left( \frac{2(p+1)!(F(x^0) - F(x^\star))}{(\alpha - 1)H_0(K - \overline{\cU})} \right)^{\frac{p-1}{p+1}}.
    \end{align*}
    Similarly, after at most
     \begin{align*}
        K = \overline{\cU} + \left\lceil \frac{2(p + 1)!(F(x^0) - F(x^\star))}{(\alpha - 1)H_0}\left(\frac{2\alpha H_{\max}}{(p-1)! \epsilon_H} \right)^{\frac{p+1}{p-1}}\right\rceil
    \end{align*}
    iterations, it holds that $\lambda_{\min}\big(\nabla^2 F(x^{\cS(j(\min))})\big) \ge -\epsilon_H$. The proof is completed.
\end{proof}

\section{Practical Enhancements}\label{section:practical enhancements}
In Algorithm \ref{alg:ac pth-order method}, the adaptive parameter $M_k$ is updated using information from all past iterates. This strategy ensures that $M_k$ serves as a nondecreasing and close estimate of the true global Lipschitz constant $L_p$, which is key to guaranteeing convergence theoretically. However, this approach can be inefficient at times in practice. The global Lipschitz constant is often a conservative estimate, as the objective function may exhibit a much smaller local Lipschitz constant in certain regions. In such cases, relying on the global estimate prevents \texttt{HAR} from taking more aggressive steps. Besides, the maximum of all historical estimates rule is sensitive to outliers. If \texttt{HAR} encounters a single ill-conditioned point at some iteration $k$, the local Lipschitz estimate $H_k$ can become excessively large. Since the adaptive parameter $M_k$ never decreases, this single event can cause the regularization parameter to stagnate at a large value, making all subsequent steps conservative. 

A remedy is to introduce a predefined budget to limit the number of historical estimates the algorithm stores. Specifically, we maintain a history of at most $\cB$ local estimates, allowing \texttt{HAR} to periodically forget older information. This mechanism prevents a single large estimate from an ill-conditioned iterate and enables the scheme to adapt to regions where the local Lipschitz constant may be smaller. This motivates our cyclic history-aware adaptive regularization (\texttt{HAR-C}) method. In this scheme, we maintain a history of local estimates up to a budget of size $\cB$. Once the number of stored local estimates reaches this budget, we perform a refresh by clearing the history. We reset the adaptive parameter $M_k$ to the maximum of its initial guess and the current local estimate, thereby discarding outdated information as the algorithm progresses toward the optimum.
\begin{algorithm}[h]
  \caption{Cyclic History-Aware Adaptive Regularization Method}
  \label{alg:cyclic ac pth-order method}
  \begin{algorithmic}[1]
    \item \textbf{Input:} Initial point $x^0 \in \dom(\psi)$, initial guess $H_0 = M_0 >0$, parameter $\alpha>1$, budget parameter $\cB \in \mathbb{N} \cup \{\infty\}$
    \item \textbf{For} $k=1,2,\ldots$ \textbf{do}
    \item \quad \textbf{If} $k \bmod \cB = 0$ \textbf{then}
    \item \quad \quad Refresh the adaptive parameter $M_k = \max\{H_0, H_{k-1}\}$;
    \item \quad \textbf{Else}
    \item \quad \quad Update the adaptive parameter $M_k = \max\{M_{k-1}, H_{k-1}\}$;
    \item \quad \textbf{End If}
    \item \quad Calculate the regularization parameter $\sigma_k = \alpha M_k$;
    \item \quad Compute the search direction 
    \begin{align*}
       d^k = \argmin_{d \in \dom(\psi)} \Omega_p^{\sigma_k}(x^{k-1} + d; x^{k-1}) + \psi(x^{k-1} + d);   
    \end{align*}
    \item \quad Update the intermediate iterate $x^{k - 0.5} = x^{k-1} + d^k$; 
    \item \quad Update the iterate $x^k = \argmin_{x\in\{x^{k-0.5}, x^{k-1}\}} F(x)$; 
    \item \quad Update the local Lipschitz estimate 
    \begin{align*}
       \frac{H_k}{(p+1)!} = \frac{f(x^{k - 0.5}) - \T_p(x^{k-1} + d^k; x^{k-1})}
        {\|d^k\|^{p+1}}; 
    \end{align*} 
    \item \textbf{End For}
  \end{algorithmic}
\end{algorithm}

From the analysis of \texttt{HAR} in Section \ref{section:HA-AR}, a key element is the finiteness of the unsuccessful index set. This property and the monotone step together ensure that function value decrease is established during the successful iterations, even though they may be interspersed with unsuccessful ones. Intuitively, to ensure convergence, the budget $\cB$ in \texttt{HAR-C} should be larger than $\overline{\cU}$. The trade-off that naturally arises is that \texttt{HAR-C} is no longer parameter-free, as this requirement necessitates a rough upper estimate of the global Lipschitz constant. Let $\overline{L_p}$ be an upper bound of the true Lipschitz constant $L_p$. If the budget $\cB$ satisfies the following condition:
\begin{equation}\label{eq:requirement of budget}
    \cB \ge \left\lceil\log _{\frac{\alpha + 1}{2}}\frac{\max\{\overline{L_p}, H_0\}}{H_0}\right\rceil + 1,
\end{equation}
then \texttt{HAR-C} is guaranteed to not get stuck. Recall that in the proofs of Theorems \ref{thm:HA-AR convex} and \ref{thm:HA-AR nonconvex}, we first establish the convergence behavior in terms of the successful iterates (see Equation \eqref{eq:convergence of successful iterates} and Lemma \ref{lemma:relate d with gradient Hessian}), and then translate this behavior to the entire sequence of iterates. Note that the convergence behavior of the successful iterates is not affected by the cyclic update introduced in \texttt{HAR-C}. Therefore, the key difference compared with \texttt{HAR} is the bound on the number of null iterations. In \texttt{HAR-C}, up to $\overline{\cU}$ unsuccessful iterations can occur in each cycle.
\begin{lemma}\label{lemma:bounded unsuccessful index set with cycle}
    Suppose that function $f$ satisfies \eqref{eq:Lipchitz continuous}. The budget satisfies condition \eqref{eq:requirement of budget}. The cardinality of the unsuccessful index set $\cU$ in Algorithm \ref{alg:cyclic ac pth-order method} satisfies
    \begin{align*}
        |\cU \cap \{1, \ldots, k\}| \le \left\lceil \frac{k}{\cB}\right\rceil \overline{\cU}, \ \forall k \ge 1.
    \end{align*}
\end{lemma}
\begin{proof}
   Since the budget $\cB$ satisfies condition \eqref{eq:requirement of budget}, by Lemma \ref{lemma:bounded unsuccessful index set}, there exists at least one successful iteration in the cycle, and thus \texttt{HAR-C} does not get stuck. Let $N_{\operatorname{cyc}}(k) = \lceil k / \cB\rceil$ be the current cycle number at iteration $k$. From Lemma \ref{lemma:bounded unsuccessful index set}, we know that at most $\overline{\cU}$ unsuccessful iterations can occur within each cycle. Therefore, the total number of unsuccessful iterations up to iteration $k$ is upper bounded by $N_{\operatorname{cyc}}(k) \overline{\cU}$, which completes the proof.
\end{proof}

\begin{theorem}\label{thm:CHA-AR convex}
    Suppose the function $f$ is convex and satisfies \eqref{eq:Lipchitz continuous}. The budget satisfies condition \eqref{eq:requirement of budget}. Furthermore, assume that the initial sublevel set $\cF(x^0) \coloneq \{x \in \dom(\psi) \mid F(x) \le F(x^0)\}$ is bounded, i.e., $R \coloneq \sup_{x,y \in \cF(x^0)} \|x - y\| < \infty$. Then the sequence $\{x^k\}_{k \ge 0}$ generated by Algorithm~\ref{alg:cyclic ac pth-order method} satisfies
    \begin{align*}
        F(x^{k}) - F(x^\star) \le \frac{2\alpha H_{\max}R^{p+1}}{p!}\left(\frac{p+1}{k(1 - \overline{\cU} / \cB) - \overline{\cU}}\right)^p, \ \forall k \ge \frac{\overline{\cU}}{1 - \overline{\cU}/\cB}. 
    \end{align*}
\end{theorem}
\begin{proof}
    From Equation \eqref{eq:convergence of successful iterates}, we know the convergence behavior in terms of the successful iterates is
   \begin{align*}
        F(x^{\cS(j)}) - F(x^\star) \le \frac{2\alpha H_{\max}R^{p+1}}{p!}\left(\frac{p+1}{j}\right)^p, \ \forall j \ge 1. 
   \end{align*}
   Besides, note that for any iteration count $k$, it holds that $k = |\cS \cap \{1, \ldots, k\}| +  |\cU \cap \{1, \ldots, k\}|$. Let $j(k) \coloneqq |\cS \cap \{1, \ldots, k\}|$. Then from Lemma \ref{lemma:bounded unsuccessful index set with cycle}, we have
   \begin{align*}
     j(k) = k - |\cU \cap \{1, \ldots, k\}| \ge k - \left\lceil \frac{k}{\cB}\right\rceil \overline{\cU} \ge k\left(1 - \frac{\overline{\cU}}{\cB}\right) - \overline{\cU}.
   \end{align*}
   Note that $F(x^{k}) - F(x^\star) \le F(x^{\cS(j(k))}) - F(x^\star)$ due to the monotonicity of \texttt{HAR-C}. Therefore, we conclude
    \begin{align*}
         F(x^{k}) - F(x^\star) \le \frac{2\alpha H_{\max}R^{p+1}}{p!}\left(\frac{p+1}{j(k)}\right)^p \le \frac{2\alpha H_{\max}R^{p+1}}{p!}\left(\frac{p+1}{k(1 - \overline{\cU} / \cB) - \overline{\cU}}\right)^p, \ \forall k \ge \frac{\overline{\cU}}{1 - \overline{\cU}/\cB}. 
    \end{align*}
    The proof is completed.
\end{proof}
\begin{corollary}
    Suppose the function $f$ is convex and satisfies \eqref{eq:Lipchitz continuous}. The budget satisfies condition \eqref{eq:requirement of budget}. Furthermore, assume that the initial sublevel set $\cF(x^0) \coloneq \{x \in \dom(\psi) \mid F(x) \le F(x^0)\}$ is bounded, i.e., $R \coloneq \sup_{x,y \in \cF(x^0)} \|x - y\| < \infty$. Let the sequence $\{x^k\}_{k \ge 0}$ be generated by Algorithm~\ref{alg:ac pth-order method}. Then for any accuracy $\epsilon > 0$, Algorithm \ref{alg:cyclic ac pth-order method} requires at most
    \begin{align*}
        K = \frac{\overline{\cU}}{1 - \overline{\cU}/\cB} + \mathcal{O}\left(\frac{1}{1 - \overline{\cU}/\cB}\left(\frac{\alpha H_{\max} R^{p+1}}{\epsilon}\right)^{\frac{1}{p}}\right)
    \end{align*}
    iterations to return an iterate $x^K$ such that $F(x^K) - F(x^\star) \le \epsilon$.
\end{corollary}

\begin{theorem}\label{thm:CHA-AR nonconvex}
   Suppose that function $f$ satisfies \eqref{eq:Lipchitz continuous}. The budget satisfies condition \eqref{eq:requirement of budget}. Let the sequence $\{x^k\}_{k \ge 0}$ be generated by Algorithm~\ref{alg:cyclic ac pth-order method}. Then for any accuracies $\epsilon_g, \epsilon_H > 0$, the algorithm requires at most
   \begin{align*}
       K = \frac{\overline{\cU}}{1 - \overline{\cU}/\cB} + \cO\left(\frac{F(x^0) - F(x^\star)}{(\alpha - 1)(1 - \overline{\cU}/\cB)H_0} \cdot \max\left\{\left(\frac{\alpha H_{\max}}{ \epsilon_g} \right)^{\frac{p+1}{p}}, \left(\frac{\alpha H_{\max}}{\epsilon_H} \right)^{\frac{p+1}{p-1}}\right\}\right)
   \end{align*}
   iterations to return an iterate $\Tilde{x}$ such that $\dist(0, \partial F(\Tilde{x})) \le \epsilon_g$. Furthermore, if the nonsmooth term vanishes, i.e, $\psi \equiv 0$, the sequence additionally satisfies $\lambda_{\min}(\nabla^2 F(\Tilde{x})) \ge -\epsilon_{H}$.
\end{theorem}
\begin{proof}
    Since the introduced cycle does not affect the convergence behavior of the successful iterates, Equation \eqref{eq:d min nonconve inequality} still holds. That is
    \begin{align*}
        \|d^{\cS(j(\min))}\|^{p+1} \coloneq \min_{j = 1, \ldots, J} \|d^{\cS(j)}\|^{p+1} \le \frac{2(p+1)!(F(x^0) - F(x^\star))}{(\alpha - 1)H_0 J}.  
    \end{align*}
    Clearly, $\cS(j(\min)) \in \{\cS(1), \ldots, \cS(J)\}$. Let $K = \cS(J)$. Again, we use Lemma \ref{lemma:bounded unsuccessful index set with cycle} to get
    \begin{align*}
        K = J + |\cU \cap \{1, \ldots, K\}| \le J + \left\lceil \frac{K}{\cB}\right\rceil \overline{\cU} \le J + \left(\frac{K}{\cB} + 1\right)\overline{\cU}.
    \end{align*}
    This implies $J \ge K\left(1 - \frac{\overline{\cU}}{\cB}\right) - \overline{\cU}$. Substituting this inequality back and using Lemma \ref{lemma:relate d with gradient Hessian} completes the proof.
\end{proof}

\begin{algorithm}[h]
  \caption{Sliding-Window History-Aware Adaptive Regularization Method}
  \label{alg:slide ac pth-order method}
  \begin{algorithmic}[1]
    \item \textbf{Input:} Initial point $x^0 \in \dom(\psi)$, initial guess $H_0 = M_0 >0$, parameter $\alpha>1$, budget parameter $\cB \in \mathbb{N} \cup \{\infty\}$
    \item \textbf{For} $k=1,2,\ldots$ \textbf{do}
    \item \quad Calculate $b_k = \max\{1, k - \cB\}$;
    \item \quad Update the adaptive parameter $M_k = \max\{H_0, H_{b_k}, H_{b_k + 1}, \ldots, H_{k-1}\}$;
    \item \quad Calculate the regularization parameter $\sigma_k = \alpha M_k$;
    \item \quad Compute the search direction 
    \begin{align*}
       d^k = \argmin_{d \in \dom(\psi)} \Omega_p^{\sigma_k}(x^{k-1} + d; x^{k-1}) + \psi(x^{k-1} + d);   
    \end{align*}
    \item \quad Update the intermediate iterate $x^{k - 0.5} = x^{k-1} + d^k$; 
    \item \quad Update the iterate $x^k = \argmin_{x\in\{x^{k-0.5}, x^{k-1}\}} F(x)$; 
    \item \quad Update the local Lipschitz estimate 
    \begin{align*}
       \frac{H_k}{(p+1)!} = \frac{f(x^{k - 0.5}) - \T_p(x^{k-1} + d^k; x^{k-1})}
        {\|d^k\|^{p+1}}; 
    \end{align*} 
    \item \textbf{End For}
  \end{algorithmic}
\end{algorithm}

As an alternative to the cyclic refresh strategy, we introduce the sliding-window history-aware adaptive regularization (\texttt{HAR-S}) method in Algorithm \ref{alg:slide ac pth-order method}. This approach maintains a sliding window, using only the $\cB$ most recent local estimates to update the adaptive parameter. Theoretically, \texttt{HAR-S} converges as long as the choice of $\cB$ satisfies the condition \eqref{eq:requirement of budget}. We first establish an upper bound on the cardinality of unsuccessful iterates in \texttt{HAR-S} that is identical to the one in Lemma \ref{lemma:bounded unsuccessful index set with cycle}.
\begin{lemma}\label{lemma:bounded unsuccessful index set with slide}
    Suppose that function $f$ satisfies \eqref{eq:Lipchitz continuous}. The budget satisfies condition \eqref{eq:requirement of budget}. The cardinality of the unsuccessful index set $\cU$ in Algorithm \ref{alg:slide ac pth-order method} satisfies
    \begin{align*}
        |\cU \cap \{1, \ldots, k\}| \le \left\lceil \frac{k}{\cB}\right\rceil \overline{\cU}, \ \forall k \ge 1.
    \end{align*}
\end{lemma}
\begin{proof}
    We first claim that for any window $W \subseteq \mathbb{N}$ with length $\cB$, it contains at most $\overline{\cU}$ unsuccessful iterations. Choose two consecutive unsuccessful iterations $\cU(j), \cU(j+1) \in W$. Then it holds that $\cU(j+1) - \cU(j) \le \cB-1$, which implies
    \begin{align*}
        \cU(j) \in [\cU(j+1) - \cB + 1, \cU(j+1) - 1].
    \end{align*}
    Since $b_{\cU(j+1)} = \max\{1, \cU(j+1) - \cB\} \le \cU(j+1) - \cB + 1$, we have the interval inclusion
    \begin{align*}
        [\cU(j+1) - \cB + 1, \cU(j+1) - 1] \subseteq [b_{\cU(j+1)}, \cU(j+1) - 1].
    \end{align*}
    This yields $\cU(j) \in [b_{\cU(j+1)}, \cU(j+1) - 1]$, and thus $H_{\cU(j)} \le M_{\cU(j+1)}$ follows. Besides, $H_{\cU(j+1)} > (\alpha+1)M_{\cU(j+1)} / 2$ holds since $\cU(j+1)$ is an unsuccessful iteration. As a result, we obtain $H_{\cU(j+1)} > (\alpha+1)H_{\cU(j)} / 2$. Suppose there exists $J$ unsuccessful iterations inside window $W$. Without loss of generality, order these unsuccessful iterations as $\cU(j) < \cU(j+1) < \ldots < \cU(j+J-1)$. By recursively using the previous inequality, we have
    \begin{align*}
        H_{\cU(j+J-1)} \ge \left(\frac{\alpha + 1}{2}\right)^{J-1}H_{\cU(j)}.
    \end{align*}
    Similarly, since $\cU(j)$ is an unsuccessful iteration, it also holds that $H_{\cU(j)} > (\alpha + 1)M_{\cU(j)}/2 \ge (\alpha + 1)H_0/2$. Consequently, we obtain
    \begin{align*}
        H_{\cU(j+J-1)} \ge \left(\frac{\alpha + 1}{2}\right)^{J}H_{0}.
    \end{align*}
    Clearly, it follows that $J \le \overline{\cU}$. Partition $\{1, \ldots, k\}$ into $m \coloneq \lceil k /\cB \rceil$ disjoint blocks
    \begin{align*}
        \cJ_t \coloneq \{(t-1)\cB + 1, \ldots, \min\{t\cB, k\}\}, \ t = 1, \ldots, m,
    \end{align*}
    each of length is smaller than $\cB$. Therefore, summing over the disjoint blocks implies
    \begin{align*}
        |\cU \cap \{1, \ldots, k\}| = \sum_{t=1}^m |\cU \cap \cJ_t| \le \sum_{t=1}^m \overline{\cU} = \left\lceil \frac{k}{\cB}\right\rceil \overline{\cU}.
    \end{align*}
    The proof is completed.
\end{proof}

Based on the above lemma, one can establish the convergence results of \texttt{HAR-S} by using arguments similar to those in Theorems \ref{thm:CHA-AR convex} and \ref{thm:CHA-AR nonconvex}. For completeness, we provide the results as follows; the proof is omitted.
\begin{theorem}\label{thm:SHA-AR convex}
    Suppose the function $f$ is convex and satisfies \eqref{eq:Lipchitz continuous}. The budget satisfies condition \eqref{eq:requirement of budget}. Furthermore, assume that the initial sublevel set $\cF(x^0) \coloneq \{x \in \dom(\psi) \mid F(x) \le F(x^0)\}$ is bounded, i.e., $R \coloneq \sup_{x,y \in \cF(x^0)} \|x - y\| < \infty$. Then the sequence $\{x^k\}_{k \ge 0}$ generated by Algorithm~\ref{alg:slide ac pth-order method} satisfies
    \begin{align*}
        F(x^{k}) - F(x^\star) \le \frac{2\alpha H_{\max}R^{p+1}}{p!}\left(\frac{p+1}{k(1 - \overline{\cU} / \cB) - \overline{\cU}}\right)^p, \ \forall k \ge \frac{\overline{\cU}}{1 - \overline{\cU}/\cB}. 
    \end{align*}
\end{theorem}

\begin{theorem}\label{thm:SHA-AR nonconvex}
   Suppose that function $f$ satisfies \eqref{eq:Lipchitz continuous}. The budget satisfies condition \eqref{eq:requirement of budget}. Let the sequence $\{x^k\}_{k \ge 0}$ be generated by Algorithm~\ref{alg:slide ac pth-order method}. Then for any accuracies $\epsilon_g, \epsilon_H > 0$, the algorithm requires at most
   \begin{align*}
       K = \frac{\overline{\cU}}{1 - \overline{\cU}/\cB} + \cO\left(\frac{F(x^0) - F(x^\star)}{(\alpha - 1)(1 - \overline{\cU}/\cB)H_0} \cdot \max\left\{\left(\frac{\alpha H_{\max}}{ \epsilon_g} \right)^{\frac{p+1}{p}}, \left(\frac{\alpha H_{\max}}{\epsilon_H} \right)^{\frac{p+1}{p-1}}\right\}\right)
   \end{align*}
   iterations to return an iterate $\Tilde{x}$ such that $\dist(0, \partial F(\Tilde{x})) \le \epsilon_g$. Furthermore, if the nonsmooth term vanishes, i.e, $\psi \equiv 0$, the sequence additionally satisfies $\lambda_{\min}(\nabla^2 F(\Tilde{x})) \ge -\epsilon_{H}$.
\end{theorem}

\begin{remark}
    Note that when $\cB = \infty$, both methods \texttt{HAR-C} and \texttt{HAR-S} reduce to \texttt{HAR}. In this case, the iteration complexity results established above match those in Theorem \ref{thm:HA-AR convex} and Corollary \ref{corollary:HA-AR convex}. While introducing the budget is expected to worsen the worst-case complexity guarantee, the numerical experiments in Section \ref{sec.experiments} show the practical superiority of using partial historical information.
\end{remark}

\section{Acceleration for Convex Objectives}
In this section, we develop an accelerated history-aware adaptive regularization method (\texttt{HAR-A}), presented in Algorithm \ref{alg:acc HA-AR}, for convex objectives. The overall algorithm design is motivated by the original accelerated $p$th-order tensor method proposed in \cite{nesterov2021implementable}. We maintain four sequences in \texttt{HAR-A}: $\{x^k\}_{k \ge 0}$, $\{v^k\}_{k \ge 0}$, and $\{y^k\}_{k \ge 0}$ are the sequences of iterates, and $\{s^k\}_{k \ge 0}$ is used to accumulate past derivative information. It is well-known that Nesterov's acceleration can render optimization methods non-monotone (e.g., the Nesterov accelerated gradient method). Therefore, in the development of \texttt{HAR-A}, we introduce a repeat mechanism (Line~\ref{ref.line repeat} of Algorithm \ref{alg:acc HA-AR}) to ensure that all sequences remain properly synchronized and updated.
% \begin{algorithm}[!h]
% \caption{Accelerated history-aware adaptive regularization method}\label{alg.accelerated_ac_cubic}
% \KwIn{Initial point $x^0=v^0 \in \dom(\psi)$, $s^0=0$, $M_0=L_0>0$, $\alpha>1$, $A_k,a_k$ as in \eqref{eq.update a and A}}
% \For{$k = 0,1,\ldots,K$}{
%     $y^k = \frac{A_k}{A_k+a_k}x^k + \frac{a_k}{A_k+a_k}v^k$\;
%     \Repeat{}{
%         $\sigma_k = \alpha M_k$\;  
%         $d^k = \argmin_{d \in \dom(\psi)} \Omega_p^{\sigma_k}(y^{k} + d; y^{k}) + \psi(y^{k} + d)$\;  
%         $x^{k+1} = y^k + d^k$, $g^{k+1}:= \nabla f(x^{k+1}) - \nabla_d \Omega_p^{\sigma_k}(y^{k} + d^k; y^{k}) \in \partial F(x^{k+1})$\;
%         $\frac{H_{k+1}}{p!} = 
%             \frac{\left\|g^{k+1}+\frac{\sigma_k}{p!}\|d^k\|^{p-1}d^k \right\|}
%             {\|d^k\|^p}$\;
%         \If{$H_{k+1} \geq \beta M_k$}{
%             $M_k = H_{k+1}$\;
%         }
%         \Else{
%             break\;
%         }
%     }
%     $M_{k+1} = \max\{M_k,H_{k+1}\}$\;
%     $s^{k+1} = s^k + a_k g^{k+1}$\;
%     $v^{k+1} = v^0 - \frac{1}{\left[(p+1)C_p(\alpha,\beta)M_k\right]^{1/p}\|s^{k+1}\|^{(p-1)/p}}s^{k+1}$\;
% }
% \end{algorithm}
\begin{algorithm}[h]
  \caption{Accelerated History-Aware Adaptive Regularization Method}
  \label{alg:acc HA-AR}
  \begin{algorithmic}[1]
    \item \textbf{Input:} Initial point $x^0 = v^0 \in \dom(\psi)$, $s^0 = 0$, initial guess $H_0 = M_0 >0$, parameter $\alpha>1$
    \item \textbf{Initialization:}  set $\beta = (\alpha + 1)/2$, set parameters $C(\alpha,\beta)$ as in \eqref{eq.Cp} and $A_k$, $a_k$ as in \eqref{eq.update a and A}
    \item \textbf{For} $k=0,1,\ldots$ \textbf{do}
    \item \quad Update the iterate $ y^k = \frac{A_k}{A_k+a_k}x^k + \frac{a_k}{A_k+a_k}v^k$;
    \item \quad \textbf{Repeat} \label{ref.line repeat}
    \item \quad \quad Calculate the regularization parameter $\sigma_k = \alpha M_k$;
    \item \quad \quad Compute the search direction
    \begin{align*}
        d^k = \argmin_{d \in \dom(\psi)} \Omega_p^{\sigma_k}(y^{k} + d; y^{k}) + \psi(y^{k} + d);   
    \end{align*}
    \item \quad \quad Update the iterate $x^{k+1} = y^k + d^k$;
    \item \quad \quad Set the subgradient $g^{k+1} \coloneq \nabla f(x^{k+1}) - \nabla \Omega_p^{\sigma_k}(y^{k} + d^k; y^{k}) \in \partial F(x^{k+1})$;
    \item \quad \quad Update the local Lipschitz estimate \label{line:acc update local estimate}
    \begin{align*}
       \frac{H_{k+1}}{p!} = 
            \frac{\|g^{k+1}+\frac{\sigma_k}{p!}\|d^k\|^{p-1}d^k\|}
            {\|d^k\|^p};
    \end{align*} 
     \item \quad \quad \textbf{If} $H_{k+1} \ge \beta M_k$ \textbf{then}
      \item \quad \quad \quad Update the adaptive parameter $M_k = H_{k+1}$;
    \item \quad \quad \textbf{Else}
    \item \quad \quad \quad \textbf{Break};
    \item \quad \quad \textbf{End If};
     \item \quad Update the adaptive parameter $M_{k+1} = \max\{M_k,H_{k+1}\}$;
      \item \quad Update  $s^{k+1} = s^k + a_k g^{k+1}$;
      \item \quad Update the iterate $v^{k+1} = v^0 - \big((p+1)C_p(\alpha,\beta)M_{k+1}\big)^{-1/p}\|s^{k+1}\|^{-(p-1)/p}s^{k+1}$; \label{line:update vk}
    \item \textbf{End For}
  \end{algorithmic}
\end{algorithm}

The theoretical analysis comes from a modified estimating sequence technique in \cite{nesterov2008accelerating}. The global convergence of \texttt{HAR-A} is guaranteed by maintaining the following two relations across iterations
\begin{subequations}\label{eq.es relation}
    \begin{align}
        \label{eq.es relation 1}
         & \phi_k^\star  \geq A_k F(x^k), \  \phi_k^\star \coloneq \min_{x\in\dom(\psi)}\phi_k(x), \\
        \label{eq.es relation 2}
         & \phi_k(x)\leq A_kF(x)+M_k C_p(\alpha,\beta)\|x-x^0\|^{p+1}, \ \forall x\in\dom(\psi).
    \end{align}
\end{subequations}
Here $\{\phi_k(x)\}_{k\geq 0}$ is a sequence of functions that approximate $F(x)$ from both above and below. As we will show later, $v^k$ is the optimal point of $\phi_k(\cdot)$, i.e., $\phi_k(v^k) = \phi_k^\star$. Besides, $\{A_k\}_{k\geq 0}$ is a sequence that measures the convergence rate of the sequence $\{x^k\}_{k\geq 0}$. As a direct result of relationship \eqref{eq.es relation}, we have
\begin{equation*}
    F(x^k)-F(x^\star)\leq \frac{M_kC_p(\alpha,\beta)\|x^0-x^\star\|^{p+1}}{A_k}, \ \forall k \ge 0.
\end{equation*}
Therefore, the global complexity directly follows from the choice of $A_k$ and $M_k$. For the upcoming analysis, we first complete the definition of $A_k,\phi_k(x)$ in the estimating sequence framework as follows:
\begin{gather}
    \label{eq.update a and A}
    A_k=k^{p+1}, \ a_k = A_{k+1}-A_k = (k+1)^{p+1}-k^{p+1},\\
    \label{eq.Cp}
    C_p(\alpha,\beta) = \frac{(p+1)^{(3p+1)/2}(p-1)^{(3p-1)/2}}{2^p(\alpha^2-\beta^2)^{(p-1)/2}}, \\
    \label{eq.update phi}
    \begin{split}
        \phi_0(x) &= M_0C_p(\alpha,\beta)\|x-x^0\|^{p+1}, \\
        \phi_{k+1}(x) &= \phi_k(x)+a_k\left (F(x^{k+1})+\langle g^{k+1},x-x^{k+1}\rangle \right ) \\
                      &\quad +C_p(\alpha,\beta)(M_{k+1}-M_k)\|x-x^0\|^{p+1}.
    \end{split}
\end{gather}
We begin with some basic properties of the estimating sequence, which are modified from Section 3 in \cite{nesterov2021implementable}.
\begin{lemma}
    \label{lem.magnitude a and A}
    For any $k \geq 0$, the sequences $\{A_k\}_{k \ge 0}$ and $\{a_k\}_{k \ge 0}$ satisfy
    \begin{equation}
        \label{eq.magnitude a and A}
        A_{k+1}^{-1}a_k^{(p+1)/p} \leq (p+1)^{(p+1)/p}.
    \end{equation}
\end{lemma}

\begin{lemma}
    \label{lem.property of phi}
    For any $k \geq 0$, the iterate $v^k$ is the unique minimizer of $\phi_k(x)$. Besides, it holds that
            \begin{equation}
                  \label{eq.phi_k phi_low}
                  \phi_k(x)\geq \phi_k^\star + \frac{M_kC_p(\alpha,\beta)}{2^{p-1}}\|x-v^k\|^{p+1}.
              \end{equation}
\end{lemma}
\begin{proof}
    First note that 
    \begin{equation}
        \label{eq.phi recursive}
        \phi_k(x) =C_p(\alpha,\beta)M_k\|x-x^0\|^{p+1}+\sum_{i=0}^{k-1}a_i\left (F(x^{i+1})+\langle \nabla g_{i+1},x-x^{i+1}\rangle \right).
    \end{equation}
    Then from the optimality condition, we have
    \begin{equation*}
        \begin{aligned}
            0 = \nabla \phi_k(x) = \ &\nabla \left(C_p(\alpha,\beta)M_k\|x-x^0\|^{p+1}\right)+ \sum_{i=0}^{k-1} a_i g_{i+1}                \\
                              = \ &(p+1)C_p(\alpha,\beta)M_k\|x-x^0\|^{p-1}(x-x^0)+\sum_{i=0}^{k-1} a_i g_{i+1}.
        \end{aligned}
    \end{equation*}
    Solving the optimality condition and observing the update rule for $s^k$ yields
    \begin{equation*}
        x = v^0 - \frac{1}{\left[(p+1)C_p(\alpha,\beta)M_k\right]^{1/p}\|s_k\|^{(p-1)/p}}s_{k},
    \end{equation*}
    which is precisely $v^k$ as defined in Line \ref{line:update vk} of Algorithm \ref{alg:acc HA-AR}. Finally, Equation \eqref{eq.phi_k phi_low} comes from the fact that $\phi_k(x)$ is uniformly convex of degree $p$.
\end{proof}

\begin{lemma}\label{lem.upper bound L}
Suppose $f$ satisfies \eqref{eq:Lipchitz continuous}. Then during the whole process of Algorithm \ref{alg:acc HA-AR}, the total number of repetitions at each iteration is bounded by $\left\lceil\log _{\beta}\frac{H_{\max}}{H_0}\right\rceil$. Therefore, $H_k$ and $M_k$ has an upper bound $H_{\max}$.
\end{lemma}
\begin{proof}
    The proof is identical to the proof of Lemma~\ref{lemma:bounded unsuccessful index set}.
\end{proof}
Now we demonstrate that the generated search direction $d_k$ in \texttt{HAR-A} satisfies the following property, which is essential for acceleration (see Corollary 1 in \cite{nesterov2021implementable}).
\begin{lemma}
    \label{lem.ratio pass}
    Suppose $f$ is convex and satisfies \eqref{eq:Lipchitz continuous}. Then for each iteration of Algorithm~\ref{alg:acc HA-AR}, it holds that
    \begin{equation}
        \label{eq.ratio pass success}
        \langle y^k-x^{k+1},g^{k+1}\rangle > D(\alpha,\beta) \frac{\|g^{k+1}\|^{(p+1)/p}}{M_k^{1/p}}, \ k \ge 0,
    \end{equation}
    where the constant $D(\alpha,\beta) \coloneq (\alpha^2-\beta^2)^{(p-1)/2p}\left(\frac{2p}{p-1}\right)\left(\frac{p-1}{p+1}\right)^{(p+1/2p)}$.
\end{lemma}
\begin{proof}
    Note that $\frac{H_{k+1}}{p!} = 
            \frac{\|g^{k+1} 
            + \frac{\sigma_k}{p!}\|d^k\|^{p-1} d^k\|}
            {\|d^k\|^p}$, and the mechanism guarantees that $H_{k+1} < \beta M_k$ in the end. Squaring both sides of the former equation gives
            \begin{equation*}
                \frac{H_{k+1}^2}{(p!)^2}\|d^k\|^{2p} = \frac{\alpha^2 M_k^2}{(p!)^2}\|d^k\|^{2p} + \frac{2}{p!}\sigma_k\|d^k\|^{p-1} \langle d^k,g^{k+1} \rangle + \|g^{k+1}\|^2.
            \end{equation*}
            Then by arranging terms, we have
            \begin{equation*}
                \begin{aligned}
                     \langle y^k-x^{k+1},g^{k+1} \rangle = \ &\frac{p!}{2\sigma_k\|d^k\|^{p-1}}\left(\frac{\alpha^2 M_k^2 -H_{k+1}^2}{(p!)^2}\|d^k\|^{2p}+\|g^{k+1}\|^2\right)\\
                    > \ &\frac{p!}{2\sigma_k}\left(\frac{(\alpha^2-\beta^2) M_k^2}{(p!)^2}\|d^k\|^{p+1}+\frac{\|g^{k+1}\|^2}{\|d^k\|^{p-1}}\right).
                \end{aligned}
            \end{equation*}
            Introduce a univariate auxiliary function $h(t) \coloneq \gamma t^{\frac{p+1}{p-1}} +\frac{\xi}{t}, \ \gamma, \ \xi, \ t>0$. It is easy to verify that 
            \begin{equation*}
                \min_{t>0} h(t) = \frac{2p}{p-1}\left(\frac{\xi(p-1)}{p+1}\right)^{\frac{p+1}{2p}}\gamma^{\frac{p-1}{2p}}.
            \end{equation*}
            By plugging in $\xi = \frac{p!\|g^{k+1}\|^2}{2\sigma_k}$ and $\gamma = \frac{(\alpha^2-\beta^2)M_k}{2\alpha p!}$, we obtai 
            \begin{equation*}
                 \langle y^k-x^{k+1},g^{k+1} \rangle > (\alpha^2-\beta^2)^{(p-1)/2p}\left(\frac{2p}{p-1}\right)\left(\frac{p-1}{p+1}\right)^{(p+1/2p)}\cdot\frac{\|g^{k+1}\|^{(p+1)/p}}{M_k^{1/p}}.
            \end{equation*}
            The proof is completed.
\end{proof}

\begin{lemma}\label{lem.estimate sequence}
    Suppose the function $f$ is convex and satisfies \eqref{eq:Lipchitz continuous}. Then for each iteration of Algorithm~\ref{alg:acc HA-AR}, it holds that
    \begin{equation}
        \label{eq.estimating sequence relation}
        A_k F(x^k)\leq \phi_k^\star  \leq \phi_k(x)\leq A_k F(x)+M_kC_p(\alpha,\beta)\|x-x^0\|^{p+1}, \ \forall x\in \dom(\psi), \ k \ge 0,
    \end{equation}
    which is a compact form of Equation \eqref{eq.es relation}.
\end{lemma}
\begin{proof}
    We prove by induction, note that $A_0 = 0$, so Equation \eqref{eq.estimating sequence relation} holds for the trivial case $i=0$. Suppose it holds for the case $i=k$. We first check Equation \eqref{eq.es relation 2} for the case $i=k+1$. Note that
    \begin{align*}
        \phi_{k+1}(x)  = \ 
        &\phi_k(x)+a_k\left(F(x^{k+1})+\langle g^{k+1},x-x^{k+1}\rangle \right) +C_p(\alpha,\beta)(M_{k+1}-M_k)\|x-x^0\|^{p+1}             \\
                       \leq \ & A_k F(x)+M_k C_p(\alpha,\beta)\|x-x^0\|^{p+1} +a_k\left(F(x^{k+1})+\langle g^{k+1},x-x^{k+1}\rangle \right)\\
                       &+C_p(\alpha,\beta)(M_{k+1}-M_k)\|x-x^0\|^{p+1}\\
                       \leq \ & A_k F(x) +a_k F(x) + C_p(\alpha,\beta)M_{k+1}\|x-x^0\|^{p+1}                                              \\
                       = \ &A_{k+1} F(x)+C_p(\alpha,\beta)M_{k+1}\|x-x^0\|^{p+1},
    \end{align*}
    where the first inequality is from the induction hypothesis. We now check Equation \eqref{eq.es relation 1} for the case $i=k+1$. First by the uniform convexity of $\phi_k$ and the induction hypothesis, it follows
    \begin{align*}
       \phi_{k+1}^\star
         = \ &\min_{x\in\dom(\psi)} \left\{ \phi_k(x) + a_k \left(F(x^{k+1}) + \langle g^{k+1}, x - x^{k+1} \rangle \right) +C_p(\alpha,\beta)(M_{k+1}-M_k)\|x-x^0\|^{p+1}\right\} \\
         \geq \ &\min_{x\in\dom(\psi)} \left\{ \phi_k(x) + a_k \left(F(x^{k+1}) + \langle g^{k+1}, x - x^{k+1} \rangle \right) \right\} \\
         \geq \ &\min_{x\in\dom(\psi)} \Bigl\{ \phi_k^\star + \frac{C_p(\alpha,\beta)M_{k}}{2^{p-1}} \|x - v^k\|^{p+1}
        + a_k \left(F(x^{k+1}) + \langle g^{k+1}, x - x^{k+1} \rangle \right) \Bigr\}                                                 \\
         \geq \ &\min_{x\in\dom(\psi)} \Bigl\{ A_k F(x^k) + \frac{C_p(\alpha,\beta)M_{k}}{2^{p-1}} \|x - v^k\|^{p+1}
        + a_k \left(F(x^{k+1}) + \langle g^{k+1}, x - x^{k+1} \rangle \right) \Bigr\}. 
    \end{align*}
    Since $F$ is convex, we immediately obtain
    \begin{align*}
        &\phi_{k+1}^\star \\
        \ge \ &\min_{x\in\dom(\psi)} \Bigl\{ A_k F(x^{k+1}) + A_k \langle g^{k+1}, x^k - x^{k+1} \rangle + \frac{C_p(\alpha,\beta)M_{k}}{2^{p-1}} \|x - v^k\|^{p+1}+ a_k \Big(F(x^{k+1}) + \langle g^{k+1}, x - x^{k+1} \rangle \Big) \Bigr\}  
    \end{align*}
    from the above equation. After some algebraic calculation, it holds that
    \begin{align*}
       \phi_{k+1}^\star \ge \ &A_{k+1} F(x^{k+1})
        + A_{k+1} \left\langle g^{k+1},
        \frac{A_k}{A_{k+1}} x^k + \frac{a_k}{A_{k+1}} v^k - x^{k+1} \right\rangle                                                               \\
         &+ \min_{x\in\dom(\psi)} \Bigl\{\frac{C_p(\alpha,\beta)M_{k}}{2^{p-1}} \|x - v^k\|^{p+1}
        + a_k \langle g^{k+1}, x - v^k \rangle \Bigr\}                                                                                \\
        = \ &A_{k+1} F(x^{k+1})
        + A_{k+1} \left\langle g^{k+1},
        \frac{A_k}{A_{k+1}} x^k + \frac{a_k}{A_{k+1}} v^k - x^{k+1} \right\rangle                                                               \\
         &- a_k^{(p+1)/p}\cdot\frac{p2^{(p-1)(p+1)/p}}{C_p^{1/p}(\alpha,\beta)(p+1)^{(p+1/p)}}\cdot\frac{ \|g^{k+1}\|^{(p+1)/p}}{M_k^{1/p}}                                                    \\
         = \ &A_{k+1} F(x^{k+1}) + A_{k+1} \langle g^{k+1}, y^k - x^{k+1} \rangle
        - a_k^{(p+1)/p}\cdot\frac{p2^{(p-1)(p+1)/p}}{C_p^{1/p}(\alpha,\beta)(p+1)^{(p+1/p)}}\cdot\frac{ \|g^{k+1}\|^{(p+1)/p}}{M_k^{1/p}}                                                                \\
        = \ &A_{k+1} F(x^{k+1}) + A_{k+1} \left(
        \langle g^{k+1}, y^k - x^{k+1} \rangle
        - A_{k+1}^{-1}a_k^{(p+1)/p} \cdot \frac{p2^{(p-1)(p+1)/p}}{C_p^{1/p}(\alpha,\beta)(p+1)^{(p+1/p)}}\cdot\frac{ \|g^{k+1}\|^{(p+1)/p}}{M_k^{1/p}} \right).
    \end{align*}
    Finally, by using Equations \eqref{eq.magnitude a and A} and \eqref{eq.ratio pass success}, we concludde
    \begin{align*}
        \phi_{k+1}^\star
         \ge \ &A_{k+1} F(x^{k+1}) + A_{k+1} \left(
        \langle g^{k+1}, y^k - x^{k+1} \rangle
        - D(\alpha,\beta)\cdot \frac{\|g^{k+1}\|^{(p+1)/p}}{M_k^{1/p}}  \right)\\
        \geq \ &A_{k+1}F(x^{k+1}).
    \end{align*}
    The proof is completed.
\end{proof}

As a direct consequence of combining Equation \eqref{eq.update a and A}, Lemma~\ref{lem.upper bound L}, and Lemma~\ref{lem.estimate sequence}, we have the following iteration complexity result of Algorithm~\ref{alg:acc HA-AR}. Note that our result is established in terms of the function value gap. By incorporating \texttt{HAR-A} into the recently proposed accumulative regularization framework \citep{ji2025high}, the convergence guarantee can be translated to a guarantee on the gradient norm, which remains $\cO(1/k^{p+1})$ and is still adaptive.
\begin{theorem}
    \label{thm.convergence outer loop}
    Suppose the function $f$ is convex and satisfies \eqref{eq:Lipchitz continuous}. Then the sequence $\{x^k\}_{k \ge 0}$ generated by Algorithm~\ref{alg:acc HA-AR} satisfies
    \begin{equation*}
        F(x^k)-F(x^\star) \leq \frac{H_{\max}C_p(\alpha,\beta)\|x^0-x^\star\|^{p+1}}{k^{p+1}}, \ \forall k\geq 0.
    \end{equation*}
\end{theorem}
\begin{remark}
    The above iteration complexity $\cO(1/k^{p+1})$ is not optimal when $p \ge 2$ compared with the lower bound established in \cite{garg2021near}. However, to achieve the optimal iteration complexity, one generally requires some highly nontrivial line-search procedure (see \cite{monteiro2013accelerated,kovalev2022first,carmon2022optimal}), which is hard to implement. As a result, we leave the development of an easily implementable, adaptive optimal pth-order tensor method for future work.
\end{remark}

\begin{remark}\label{remark.practical} 
It is known that accelerated methods often underperform in practice due to the lack of local convergence. Their stronger worst-case complexity guarantees often come at the expense of local efficiency.
Empirical evidence indicates that in many problems, such as logistic regression, accelerated methods cannot match unaccelerated counterparts when searching for solutions in the high-precision regimes; see ~\cite{jiang2020unified,huang2024inexact,chen2022accelerating,carmon2022optimal}. Consequently, it was often necessary to manually switch to Newton’s method to achieve high-accuracy solutions.
\end{remark}

\section{Numerical Experiments}\label{sec.experiments}

We implement two practical history-aware methods using second-order derivatives, i.e., $p=2$,  and compare them to the adaptive cubic regularized Newton method.  In regard of Remark~\ref{remark.practical}, we do not provide an implementation of Algorithm~\ref{alg:acc HA-AR}.
\begin{itemize}[leftmargin=*]
    \item \harc{} (Algorithm~\ref{alg:cyclic ac pth-order method}) and \hars{} (Algorithm~\ref{alg:slide ac pth-order method}), each tested with different budget sizes $\mathcal B$ (cf. \eqref{eq:requirement of budget}).
    The resulting cubic subproblems are solved using the textbook one-dimensional line-search strategy as described in \cite{nesterov2006cubic}, endowed with Cholesky factorization to solve the arising linear systems.
    Because these subproblems are solved through direct factorizations, we do not employ inexact termination techniques commonly used in iterative Krylov solvers; see, e.g., \cite{curtis2021trust,zhang2025homogeneous}.
    \item \arc{}: The adaptive cubic regularized Newton method, originally proposed by \citet{nesterov2006cubic} and \citet{cartis2011adaptive1}.
    A scalable version is available in \cite{dussaultScalableAdaptiveCubic2024}, implemented as the open-source package \texttt{AdaptiveRegularization.jl}.
We adopt its default parameter settings, including its built-in subproblem solver.
\end{itemize}
All methods use exact Hessian evaluations to ensure fair and consistent comparison. 

\subsection{Convex optimization}
We conduct experiments on the regularized logistic regression problem, which is defined as follows.
\begin{equation}
    \label{eq.logistic regression}
    f(x) = \frac{1}{N} \sum_{i=1}^N \log \left(1 + \exp \left(-b_i a_i^\top  x\right)\right) + \frac{\gamma}{2}\|x\|^2,
\end{equation}
where $a_i \in \mathbb{R}^n$ and $b_i \in \{-1, 1\}$, $\gamma = 10^{-5}$ is the regularization parameter. We report the number of Hessian evaluations required by each method on several LIBSVM datasets.\footnote{For details, see \url{https://www.csie.ntu.edu.tw/cjlin/libsvmtools/datasets/}.}

From the results, \harc{}, \hars{}, and \arc{} generally outperform the other methods across different instances, and overall, their performances are comparable.
Interestingly, the budget size $\mathcal {B} $ plays a crucial role in practical performance. In general, a larger $\mathcal B$ yields more stable results, whereas, as observed in all six instances, a smaller $\mathcal B$ appears to favor the sliding variant (\hars{}) significantly.

\begin{figure}[ht]
    \small
    \centering
    \subcaptionbox{\texttt{a4a}}[0.43\columnwidth]{
        \resizebox{0.43\columnwidth}{!}{\input{figs/e-logistic-a4a-k}}
    }
    \subcaptionbox{\texttt{a9a}}[0.43\columnwidth]{
        \resizebox{0.43\columnwidth}{!}{\input{figs/e-logistic-a9a-k}}
    }
    \subcaptionbox{\texttt{w4a}}[0.43\columnwidth]{
        \resizebox{0.43\columnwidth}{!}{\input{figs/e-logistic-w4a-k}}
    }
    \subcaptionbox{\texttt{w8a}}[0.43\columnwidth]{
        \resizebox{0.43\columnwidth}{!}{\input{figs/e-logistic-w8a-k}}
    }
    \subcaptionbox{\texttt{mushroom}}[0.43\columnwidth]{
        \resizebox{0.43\columnwidth}{!}{\input{figs/e-logistic-mushroom-k}}
    }
     \subcaptionbox{\texttt{splice}}[0.43\columnwidth]{
        \resizebox{0.43\columnwidth}{!}{\input{figs/e-logistic-splice-k}}
    }

    % \subcaptionbox{Runtime (seconds)}{
    %     \includegraphics[width=0.43\columnwidth]{icml2025/figs/fx-softmax-t.pdf}
    % }
    \caption{Logistic regression using the LIBSVM datasets. The number in the bracket means the budget size $\mathcal B$.}\label{fig.logistic regression}
    \normalsize
\end{figure}
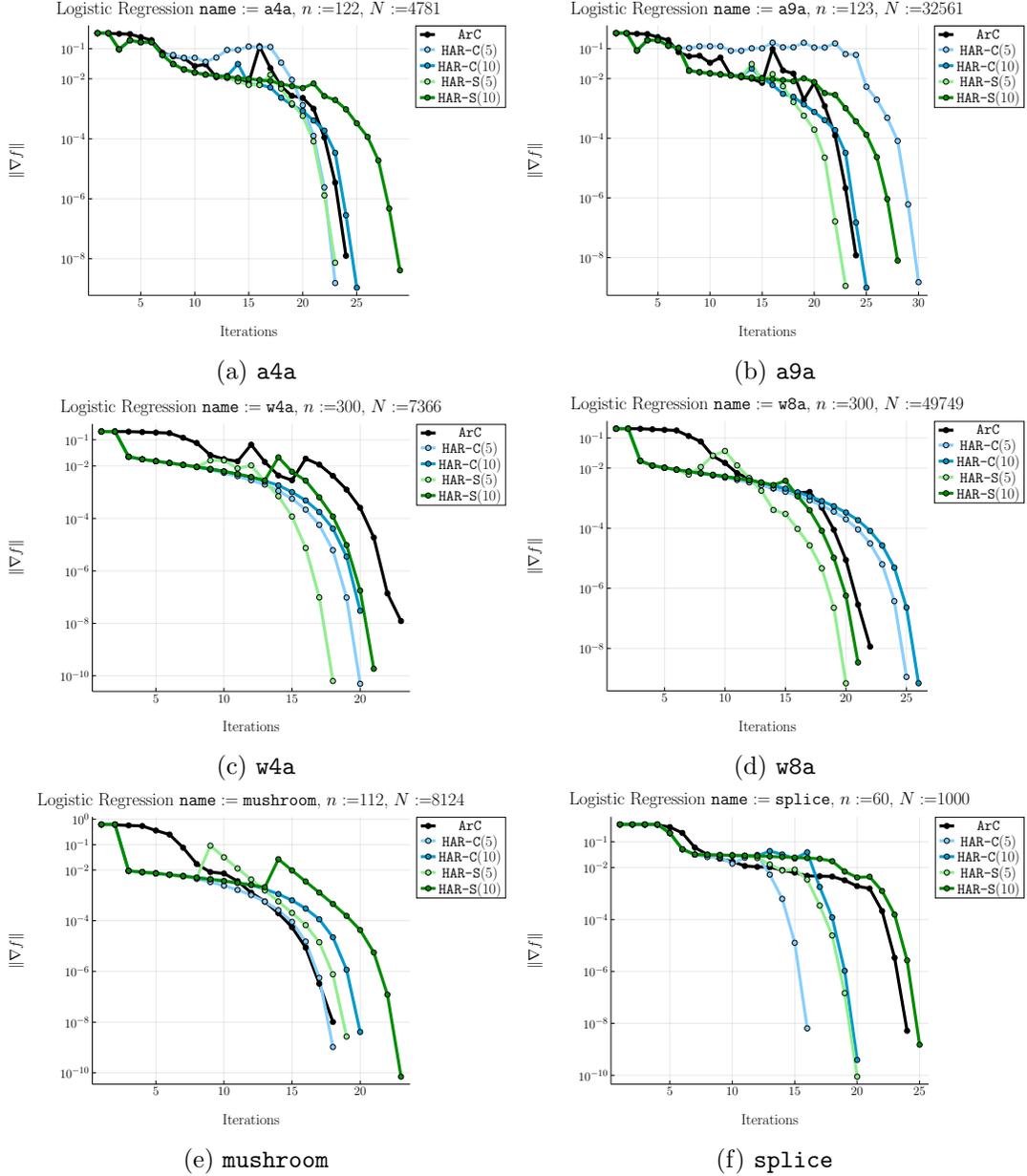

\subsection{CUTEst dataset}

We evaluate our methods against \arc{} on the CUTEst dataset; see \cite{gouldCutestConstrainedUnconstrained2015}. 
For simplicity, we restrict attention to problems with dimension $n \le 200$, yielding a total of 95 instances. 

Table~\ref{tab.perf.geocutest} summarizes the tested algorithms. As before, the number in the bracket means the budget size used in the method. Based on the knowledge we gain from convex optimization, we set the budget size to $\mathcal B = 15$ to \harc{}, while $\mathcal B = 5$ to \hars{}.
Here, $K$ denotes the number of successful instances. We report performance statistics using scaled geometric means (SGMs): $\overline t_G$, $\overline k_G^f$, $\overline k_G^g$, and $\overline k_G^H$ represent the (geometric) mean running time, and the number of function, gradient, and Hessian evaluations, respectively.

\begin{table}[h]
\centering
\caption{Performance of different algorithms on the CUTEst dataset.
$\overline t_G$ is scaled by 1 second, and $\overline k_G^f$, $\overline k_G^g$, and $\overline k_G^H$ are scaled by 50 iterations.
For failed instances, the iteration count and solving time are set to $20{,}000$.}
\label{tab.perf.geocutest}
\begin{tabular}{lrrrrr}
\toprule
Method & $K$ & $\overline t_G$ & $\overline k_G^f$ & $\overline k_G^g$ & $\overline k_G^H$ \\
\midrule
\arc{}  & 92 & 0.26 & 47.86 & 46.11 & 44.85 \\
\harc{} (15) & 91 & 0.59 & 57.34 & 58.73 & 55.95 \\
\hars{} (5) & 90 & 0.46 & 48.38 & 49.70 & 47.05 \\
\bottomrule
\end{tabular}
\end{table}

From these results, we observe that \harc{} and \hars{} are broadly comparable to the state-of-the-art \arc{} implementation. Although they solve one and two fewer instances, respectively, their geometric mean statistics are quite similar, aside from the higher running time. This slowdown primarily stems from our current cubic subproblem solver and remains an avenue for optimization.

\begin{figure}[t]
\small
\centering
\subcaptionbox{Gradient evaluations}[0.40\columnwidth]{
\includegraphics[width=0.40\columnwidth]{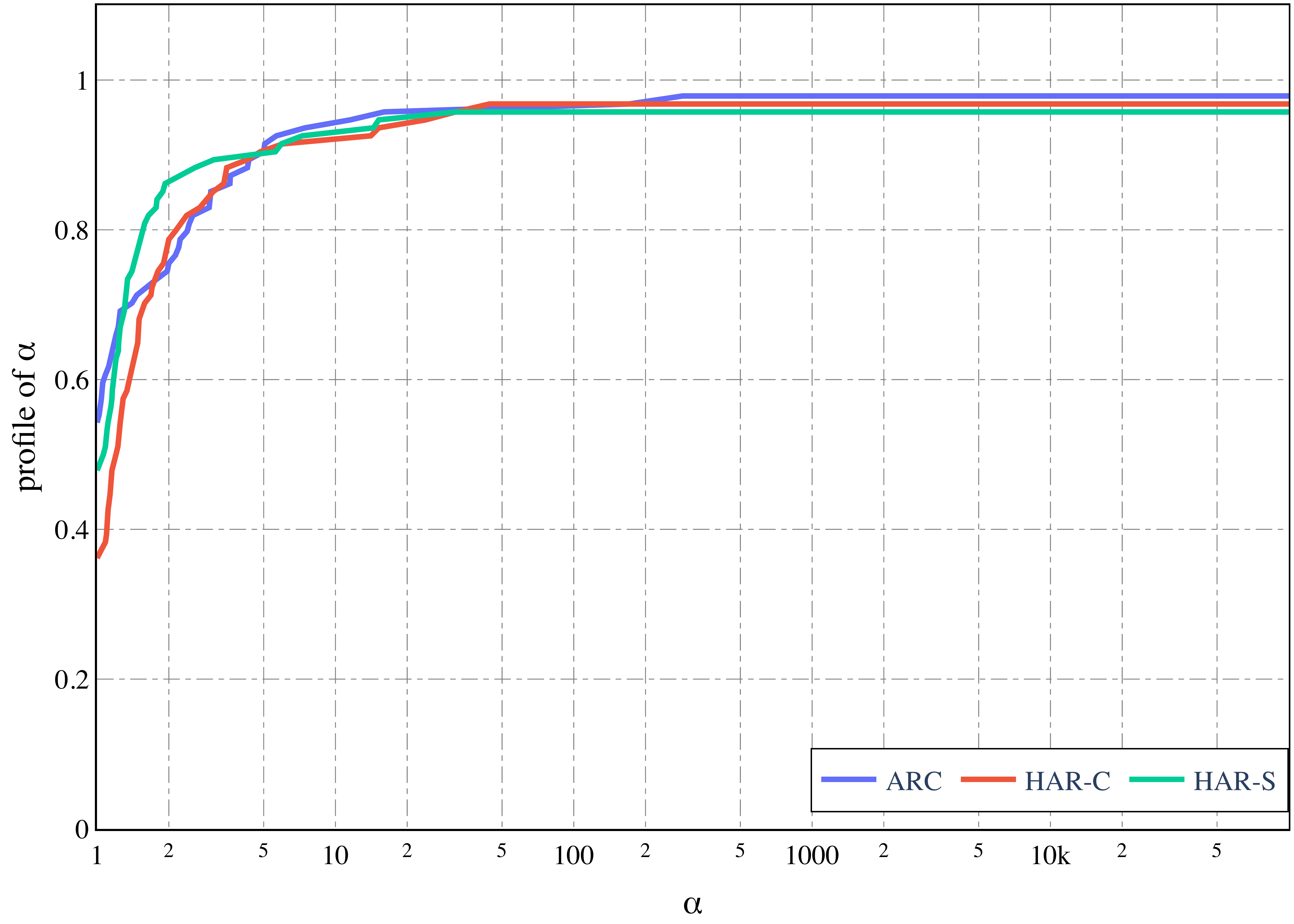}
}
\subcaptionbox{Hessian evaluations}[0.40\columnwidth]{
\includegraphics[width=0.40\columnwidth]{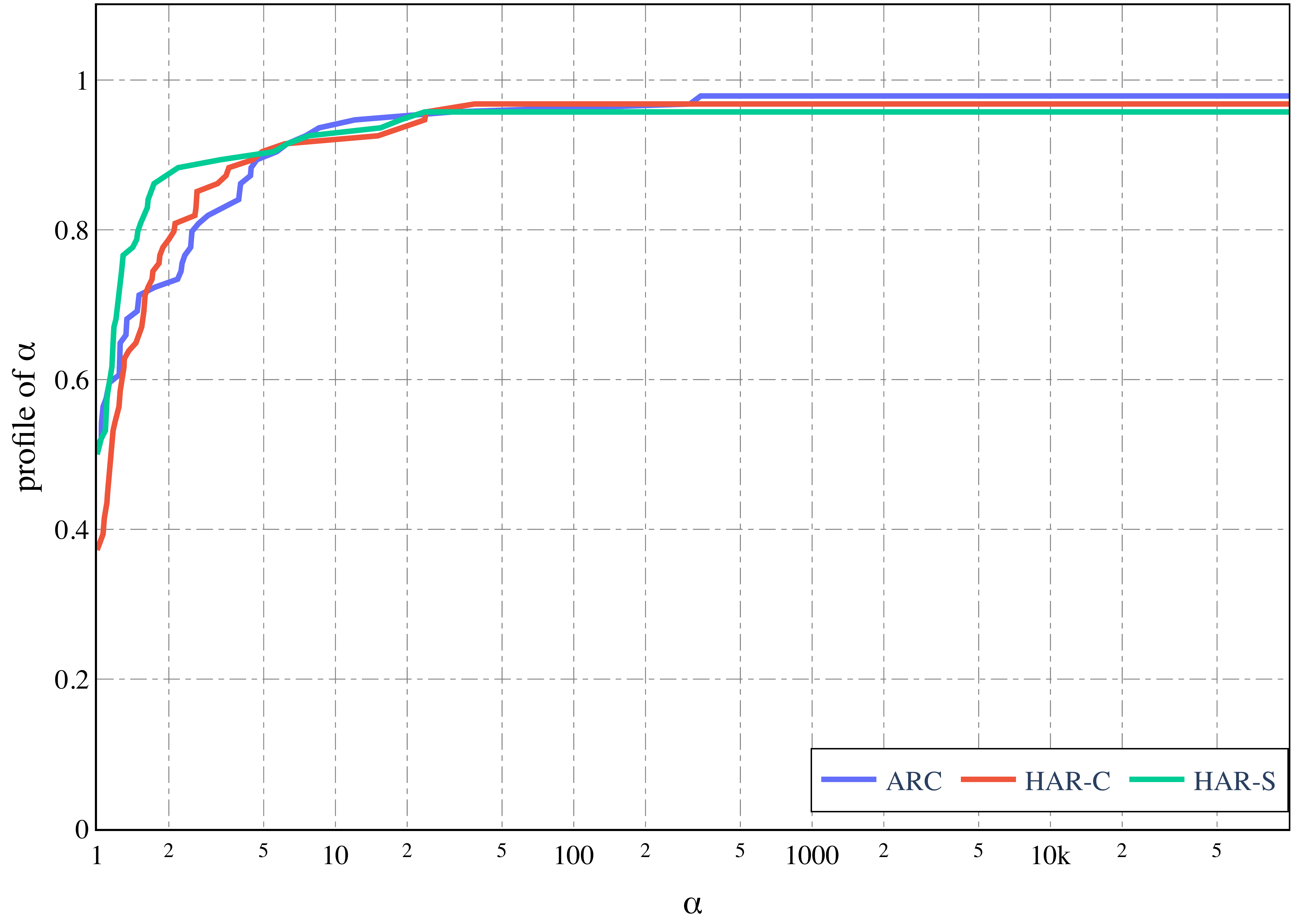}
}
\caption{Performance profiles of selected CUTEst instances. Panel (a) shows gradient evaluations; panel (b) shows Hessian evaluations.}\label{fig.perfprofile}
\normalsize
\end{figure}

Beyond SGM-based metrics, we also employ performance profiles on gradient and Hessian evaluations \citep{dolanOptimalityMeasuresPerformance2006}; see Figure~\ref{fig.perfprofile}.
At point $\alpha$, the performance profile represents the fraction of problems solved within $2^{\alpha}$ times the best iteration count among all methods.
These preliminary results suggest that both \harc{} and \hars{}, especially the latter, are highly competitive.
They motivate further exploration of practical improvements to \har{}, such as adopting a dynamic budget size $\mathcal B$.

\section{Conclusion}
In this work, we develop a new adaptive regularization strategy called the history-aware adaptive regularization method, which uses a history of local Lipschitz estimates to select the current regularization parameter. Our approach works for both convex and nonconvex objectives, with the same complexity guarantees as the standard $p$th-order tensor method that requires a known Lipschitz constant. Additionally, an Nesterov's accelerated version is proposed for convex optimization. For practical considerations, we introduce two variants that use only a portion of the past Lipschitz estimates. Given a rough upper bound on the Lipschitz constant, these two variants achieve the same iteration complexity guarantees in terms of input accuracy as the original one. In numerical experiments, these two variants demonstrate encouraging numerical performance, which validates the effectiveness of our overall approach. Future work includes investigating whether this strategy works for H\"{o}lder smoothness, and how to design an optimal version that matches the lower bound in the convex setting.

\clearpage
\bibliographystyle{abbrvnat}
\bibliography{ref}

\end{document}

%% file: header.tex
\usepackage[top=1in, bottom=1in, left=1in, right=1in]{geometry}

\usepackage[linktocpage,colorlinks,linkcolor=blue,anchorcolor=blue,citecolor=blue,urlcolor=blue,pagebackref]{hyperref}
\usepackage{euscript,mdframed}
\usepackage{microtype,todonotes,relsize}
\usepackage{amsmath,amsthm}
\usepackage{algpseudocode}

\usepackage{accents}
\usepackage{comment,url,graphicx,relsize}
\usepackage{amssymb,amsfonts,amsmath,amsthm,amscd,dsfont,mathrsfs,mathtools,nicefrac, bm}
\usepackage{float,psfrag,epsfig,color,xcolor,url}
\usepackage{epstopdf,bbm,mathtools,enumitem}
\usepackage[ruled,vlined]{algorithm2e}
\usepackage{tablefootnote}
\graphicspath{{Plots/}}
\usepackage{diagbox}
\usepackage{tikz}
\usetikzlibrary{calc,patterns,angles,quotes}
\usepackage{makecell}
\usepackage{multirow}
\usepackage{booktabs}
\usepackage[authoryear,round]{natbib}
\usepackage{adjustbox}
\usepackage{subcaption,lscape}
\usepackage{tikz}
\usepackage{pgfplots}

\usetikzlibrary{arrows.meta}
\usetikzlibrary{backgrounds}
\usepgfplotslibrary{patchplots}
\usepgfplotslibrary{fillbetween}
\pgfplotsset{%
    layers/standard/.define layer set={%
        background,axis background,axis grid,axis ticks,axis lines,axis tick labels,pre main,main,axis descriptions,axis foreground%
    }{
        grid style={/pgfplots/on layer=axis grid},%
        tick style={/pgfplots/on layer=axis ticks},%
        axis line style={/pgfplots/on layer=axis lines},%
        label style={/pgfplots/on layer=axis descriptions},%
        legend style={/pgfplots/on layer=axis descriptions},%
        title style={/pgfplots/on layer=axis descriptions},%
        colorbar style={/pgfplots/on layer=axis descriptions},%
        ticklabel style={/pgfplots/on layer=axis tick labels},%
        axis background@ style={/pgfplots/on layer=axis background},%
        3d box foreground style={/pgfplots/on layer=axis foreground},%
    },
}

%% file: macro.tex
\newtheorem{theorem}{Theorem}[section]
\newtheorem{lemma}{Lemma}[section]
\newtheorem{corollary}{Corollary}[section]

\newtheorem{remark}{Remark}[section]

\newtheorem{definition}{Definition}[section]

\newcommand{\reals}{\mathbb{R}}
\newcommand{\inner}[2]{\langle{#1},{#2}\rangle}

\newcommand{\argmin}{\operatorname{argmin}}
\newcommand{\dom}{\operatorname{dom}}
\newcommand{\dist}{\operatorname{dist}}
\newcommand{\D}{\operatorname{D}}
\newcommand{\T}{\operatorname{T}}
\newcommand{\I}{\operatorname{I}}

\newcommand{\cS}{\mathcal{S}}
\newcommand{\cB}{\mathcal{B}}

\newcommand{\cF}{\mathcal{F}}
\newcommand{\cU}{\mathcal{U}}
\newcommand{\cO}{\mathcal{O}}

\newcommand{\cJ}{\mathcal{J}}

\newcommand{\har}{\texttt{HAR}}
\newcommand{\harc}{\texttt{HAR-C}}
\newcommand{\hars}{\texttt{HAR-S}}

\renewcommand{\arc}{\texttt{ARC}}

%% file: figs/e-logistic-a4a-k.tex
% Recommended preamble:
% \usetikzlibrary{arrows.meta}
% \usetikzlibrary{backgrounds}
% \usepgfplotslibrary{patchplots}
% \usepgfplotslibrary{fillbetween}
% \pgfplotsset{%
%     layers/standard/.define layer set={%
%         background,axis background,axis grid,axis ticks,axis lines,axis tick labels,pre main,main,axis descriptions,axis foreground%
%     }{
%         grid style={/pgfplots/on layer=axis grid},%
%         tick style={/pgfplots/on layer=axis ticks},%
%         axis line style={/pgfplots/on layer=axis lines},%
%         label style={/pgfplots/on layer=axis descriptions},%
%         legend style={/pgfplots/on layer=axis descriptions},%
%         title style={/pgfplots/on layer=axis descriptions},%
%         colorbar style={/pgfplots/on layer=axis descriptions},%
%         ticklabel style={/pgfplots/on layer=axis tick labels},%
%         axis background@ style={/pgfplots/on layer=axis background},%
%         3d box foreground style={/pgfplots/on layer=axis foreground},%
%     },
% }

\begin{tikzpicture}[/tikz/background rectangle/.style={fill={rgb,1:red,1.0;green,1.0;blue,1.0}, fill opacity={1.0}, draw opacity={1.0}}, show background rectangle]
\begin{axis}[point meta max={nan}, point meta min={nan}, legend cell align={left}, legend columns={1}, title={Logistic Regression $\texttt{name}:=\texttt{a4a}$, $n:=$122, $N:=$4781}, title style={at={{(0.5,1)}}, anchor={south}, font={{\fontsize{22 pt}{28.6 pt}\selectfont}}, color={rgb,1:red,0.0;green,0.0;blue,0.0}, draw opacity={1.0}, rotate={0.0}, align={center}}, legend style={color={rgb,1:red,0.0;green,0.0;blue,0.0}, draw opacity={1.0}, line width={1}, solid, fill={rgb,1:red,1.0;green,1.0;blue,1.0}, fill opacity={1.0}, text opacity={1.0}, font={{\fontsize{16 pt}{20.8 pt}\selectfont}}, text={rgb,1:red,0.0;green,0.0;blue,0.0}, cells={anchor={center}}, at={(1.02, 1)}, anchor={north west}}, axis background/.style={fill={rgb,1:red,1.0;green,1.0;blue,1.0}, opacity={1.0}}, anchor={north west}, xshift={1.0mm}, yshift={-1.0mm}, width={145.4mm}, height={125.0mm}, scaled x ticks={false}, xlabel={$\textrm{Iterations}$}, x tick style={color={rgb,1:red,0.0;green,0.0;blue,0.0}, opacity={1.0}}, x tick label style={color={rgb,1:red,0.0;green,0.0;blue,0.0}, opacity={1.0}, rotate={0}}, xlabel style={at={(ticklabel cs:0.5)}, anchor=near ticklabel, at={{(ticklabel cs:0.5)}}, anchor={near ticklabel}, font={{\fontsize{16 pt}{20.8 pt}\selectfont}}, color={rgb,1:red,0.0;green,0.0;blue,0.0}, draw opacity={1.0}, rotate={0.0}}, xmajorgrids={true}, xmin={0.16000000000000014}, xmax={29.84}, xticklabels={{$5$,$10$,$15$,$20$,$25$}}, xtick={{5.0,10.0,15.0,20.0,25.0}}, xtick align={inside}, xticklabel style={font={{\fontsize{15 pt}{19.5 pt}\selectfont}}, color={rgb,1:red,0.0;green,0.0;blue,0.0}, draw opacity={1.0}, rotate={0.0}}, x grid style={color={rgb,1:red,0.0;green,0.0;blue,0.0}, draw opacity={0.1}, line width={0.5}, solid}, axis x line*={left}, x axis line style={color={rgb,1:red,0.0;green,0.0;blue,0.0}, draw opacity={1.0}, line width={1}, solid}, scaled y ticks={false}, ylabel={$\|\nabla f\|$}, y tick style={color={rgb,1:red,0.0;green,0.0;blue,0.0}, opacity={1.0}}, y tick label style={color={rgb,1:red,0.0;green,0.0;blue,0.0}, opacity={1.0}, rotate={0}}, ylabel style={at={(ticklabel cs:0.5)}, anchor=near ticklabel, at={{(ticklabel cs:0.5)}}, anchor={near ticklabel}, font={{\fontsize{16 pt}{20.8 pt}\selectfont}}, color={rgb,1:red,0.0;green,0.0;blue,0.0}, draw opacity={1.0}, rotate={0.0}}, ymode={log}, log basis y={10}, ymajorgrids={true}, ymin={6.050983376844389e-10}, ymax={0.597008915417402}, yticklabels={{$10^{-8}$,$10^{-6}$,$10^{-4}$,$10^{-2}$,$10^{-1}$}}, ytick={{1.0e-8,1.0e-6,0.0001,0.01,0.1}}, ytick align={inside}, yticklabel style={font={{\fontsize{15 pt}{19.5 pt}\selectfont}}, color={rgb,1:red,0.0;green,0.0;blue,0.0}, draw opacity={1.0}, rotate={0.0}}, y grid style={color={rgb,1:red,0.0;green,0.0;blue,0.0}, draw opacity={0.1}, line width={0.5}, solid}, axis y line*={left}, y axis line style={color={rgb,1:red,0.0;green,0.0;blue,0.0}, draw opacity={1.0}, line width={1}, solid}, colorbar={false}]
    \addplot[color={rgb,1:red,0.0;green,0.0;blue,0.0}, name path={76}, draw opacity={1.0}, line width={3.5}, solid, mark={*}, mark size={3.0 pt}, mark repeat={1}, mark options={color={rgb,1:red,0.0;green,0.0;blue,0.0}, draw opacity={1.0}, fill={rgb,1:red,0.0;green,0.0;blue,0.0}, fill opacity={1.0}, line width={0.75}, rotate={0}, solid}]
        table[row sep={\\}]
        {
            \\
            1.0  0.33222266993321664  \\
            2.0  0.32403563902334165  \\
            3.0  0.31523419914282896  \\
            4.0  0.30743870979762905  \\
            5.0  0.24321050984358208  \\
            6.0  0.19264656912344139  \\
            7.0  0.07398715664402927  \\
            8.0  0.05548805332371764  \\
            9.0  0.04854677925138574  \\
            10.0  0.02648594241577243  \\
            11.0  0.029992673851114534  \\
            12.0  0.011252296595893196  \\
            13.0  0.010879841185898619  \\
            14.0  0.008689607881529721  \\
            15.0  0.007450850028609041  \\
            16.0  0.12050563933635956  \\
            17.0  0.022704277944378896  \\
            18.0  0.0055272344909455305  \\
            19.0  0.0026880732418780953  \\
            20.0  0.002290370349390935  \\
            21.0  0.0010016565968009538  \\
            22.0  0.00011040162043910506  \\
            23.0  3.4562916129309467e-6  \\
            24.0  1.2571938467617309e-8  \\
        }
        ;
    \addlegendentry {$\texttt{ArC}$}
    \addplot[color={rgb,1:red,0.5294;green,0.8078;blue,1.0}, name path={77}, draw opacity={1.0}, line width={3.5}, solid, mark={*}, mark size={3.0 pt}, mark repeat={1}, mark options={color={rgb,1:red,0.0;green,0.0;blue,0.0}, draw opacity={1.0}, fill={rgb,1:red,0.5294;green,0.8078;blue,1.0}, fill opacity={1.0}, line width={0.75}, rotate={0}, solid}]
        table[row sep={\\}]
        {
            \\
            1.0  0.3322226699332181  \\
            2.0  0.3322226699332181  \\
            3.0  0.09583535320778457  \\
            4.0  0.19043410432664173  \\
            5.0  0.16496767990405814  \\
            6.0  0.16496767990405814  \\
            7.0  0.06216897112231976  \\
            8.0  0.06216897112231976  \\
            9.0  0.049805270763579226  \\
            10.0  0.049805270763579226  \\
            11.0  0.03665924843222731  \\
            12.0  0.050869890342117016  \\
            13.0  0.09152207993667456  \\
            14.0  0.09152207993667456  \\
            15.0  0.11837315459336248  \\
            16.0  0.11123419649197423  \\
            17.0  0.11530469916188302  \\
            18.0  0.034268396118127153  \\
            19.0  0.009182518143137962  \\
            20.0  0.0013106584921834712  \\
            21.0  0.00012448460092570683  \\
            22.0  2.395373100847094e-6  \\
            23.0  1.5539562666610145e-9  \\
        }
        ;
    \addlegendentry {$\texttt{HAR-C} (5)$}
    \addplot[color={rgb,1:red,0.0;green,0.6039;blue,0.8039}, name path={78}, draw opacity={1.0}, line width={3.5}, solid, mark={*}, mark size={3.0 pt}, mark repeat={1}, mark options={color={rgb,1:red,0.0;green,0.0;blue,0.0}, draw opacity={1.0}, fill={rgb,1:red,0.0;green,0.6039;blue,0.8039}, fill opacity={1.0}, line width={0.75}, rotate={0}, solid}]
        table[row sep={\\}]
        {
            \\
            1.0  0.3322226699332181  \\
            2.0  0.3322226699332181  \\
            3.0  0.09583535320778457  \\
            4.0  0.19043410432664173  \\
            5.0  0.16496767990405814  \\
            6.0  0.16496767990405814  \\
            7.0  0.06216897112231976  \\
            8.0  0.030820617109238325  \\
            9.0  0.020414638425556044  \\
            10.0  0.016021624114592552  \\
            11.0  0.013677771611906694  \\
            12.0  0.012493806738386407  \\
            13.0  0.012493806738386407  \\
            14.0  0.030966739932974278  \\
            15.0  0.0066075074706756195  \\
            16.0  0.006352143396694358  \\
            17.0  0.00512237650617398  \\
            18.0  0.0023249032197271416  \\
            19.0  0.0013571494428479102  \\
            20.0  0.0008354771782164257  \\
            21.0  0.00040627582939314395  \\
            22.0  0.00018650521454695006  \\
            23.0  3.40336971068741e-5  \\
            24.0  2.830439635183476e-7  \\
            25.0  1.0873704144707407e-9  \\
        }
        ;
    \addlegendentry {$\texttt{HAR-C} (10)$}
    \addplot[color={rgb,1:red,0.5647;green,0.9333;blue,0.5647}, name path={79}, draw opacity={1.0}, line width={3.5}, solid, mark={*}, mark size={3.0 pt}, mark repeat={1}, mark options={color={rgb,1:red,0.0;green,0.0;blue,0.0}, draw opacity={1.0}, fill={rgb,1:red,0.5647;green,0.9333;blue,0.5647}, fill opacity={1.0}, line width={0.75}, rotate={0}, solid}]
        table[row sep={\\}]
        {
            \\
            1.0  0.3322226699332181  \\
            2.0  0.3322226699332181  \\
            3.0  0.09583535320778457  \\
            4.0  0.19043410432664173  \\
            5.0  0.16496767990405814  \\
            6.0  0.16496767990405814  \\
            7.0  0.06216897112231976  \\
            8.0  0.030820617109238325  \\
            9.0  0.020414638425556044  \\
            10.0  0.016021624114592552  \\
            11.0  0.013677771611906694  \\
            12.0  0.012493806738386407  \\
            13.0  0.010658151661822183  \\
            14.0  0.008121210375487475  \\
            15.0  0.00623516141786392  \\
            16.0  0.006190936552105207  \\
            17.0  0.013567314178477978  \\
            18.0  0.004681534895137273  \\
            19.0  0.0014922035869013258  \\
            20.0  0.0005787282685301421  \\
            21.0  8.159083977533664e-5  \\
            22.0  1.2964194996206739e-6  \\
            23.0  7.387743161601891e-9  \\
        }
        ;
    \addlegendentry {$\texttt{HAR-S} (5)$}
    \addplot[color={rgb,1:red,0.0;green,0.5451;blue,0.0}, name path={80}, draw opacity={1.0}, line width={3.5}, solid, mark={*}, mark size={3.0 pt}, mark repeat={1}, mark options={color={rgb,1:red,0.0;green,0.0;blue,0.0}, draw opacity={1.0}, fill={rgb,1:red,0.0;green,0.5451;blue,0.0}, fill opacity={1.0}, line width={0.75}, rotate={0}, solid}]
        table[row sep={\\}]
        {
            \\
            1.0  0.3322226699332181  \\
            2.0  0.3322226699332181  \\
            3.0  0.09583535320778457  \\
            4.0  0.19043410432664173  \\
            5.0  0.16496767990405814  \\
            6.0  0.16496767990405814  \\
            7.0  0.06216897112231976  \\
            8.0  0.030820617109238325  \\
            9.0  0.020414638425556044  \\
            10.0  0.016021624114592552  \\
            11.0  0.013677771611906694  \\
            12.0  0.012493806738386407  \\
            13.0  0.011507206757811424  \\
            14.0  0.010650742173300483  \\
            15.0  0.009862364224321574  \\
            16.0  0.009070961946685784  \\
            17.0  0.008316201607849707  \\
            18.0  0.006720565584254094  \\
            19.0  0.00567321320507275  \\
            20.0  0.004889738389966419  \\
            21.0  0.0069414744315298435  \\
            22.0  0.0026343374719246804  \\
            23.0  0.0019548302149389974  \\
            24.0  0.0009551163935541683  \\
            25.0  0.00032710117332547956  \\
            26.0  0.00011555223859763291  \\
            27.0  1.8929439618802903e-5  \\
            28.0  4.7353314726237614e-7  \\
            29.0  4.122367617459e-9  \\
        }
        ;
    \addlegendentry {$\texttt{HAR-S} (10)$}
\end{axis}
\end{tikzpicture}

%% file: figs/e-logistic-a9a-k.tex
% Recommended preamble:
% \usetikzlibrary{arrows.meta}
% \usetikzlibrary{backgrounds}
% \usepgfplotslibrary{patchplots}
% \usepgfplotslibrary{fillbetween}
% \pgfplotsset{%
%     layers/standard/.define layer set={%
%         background,axis background,axis grid,axis ticks,axis lines,axis tick labels,pre main,main,axis descriptions,axis foreground%
%     }{
%         grid style={/pgfplots/on layer=axis grid},%
%         tick style={/pgfplots/on layer=axis ticks},%
%         axis line style={/pgfplots/on layer=axis lines},%
%         label style={/pgfplots/on layer=axis descriptions},%
%         legend style={/pgfplots/on layer=axis descriptions},%
%         title style={/pgfplots/on layer=axis descriptions},%
%         colorbar style={/pgfplots/on layer=axis descriptions},%
%         ticklabel style={/pgfplots/on layer=axis tick labels},%
%         axis background@ style={/pgfplots/on layer=axis background},%
%         3d box foreground style={/pgfplots/on layer=axis foreground},%
%     },
% }

\begin{tikzpicture}[/tikz/background rectangle/.style={fill={rgb,1:red,1.0;green,1.0;blue,1.0}, fill opacity={1.0}, draw opacity={1.0}}, show background rectangle]
\begin{axis}[point meta max={nan}, point meta min={nan}, legend cell align={left}, legend columns={1}, title={Logistic Regression $\texttt{name}:=\texttt{a9a}$, $n:=$123, $N:=$32561}, title style={at={{(0.5,1)}}, anchor={south}, font={{\fontsize{22 pt}{28.6 pt}\selectfont}}, color={rgb,1:red,0.0;green,0.0;blue,0.0}, draw opacity={1.0}, rotate={0.0}, align={center}}, legend style={color={rgb,1:red,0.0;green,0.0;blue,0.0}, draw opacity={1.0}, line width={1}, solid, fill={rgb,1:red,1.0;green,1.0;blue,1.0}, fill opacity={1.0}, text opacity={1.0}, font={{\fontsize{16 pt}{20.8 pt}\selectfont}}, text={rgb,1:red,0.0;green,0.0;blue,0.0}, cells={anchor={center}}, at={(1.02, 1)}, anchor={north west}}, axis background/.style={fill={rgb,1:red,1.0;green,1.0;blue,1.0}, opacity={1.0}}, anchor={north west}, xshift={1.0mm}, yshift={-1.0mm}, width={145.4mm}, height={125.0mm}, scaled x ticks={false}, xlabel={$\textrm{Iterations}$}, x tick style={color={rgb,1:red,0.0;green,0.0;blue,0.0}, opacity={1.0}}, x tick label style={color={rgb,1:red,0.0;green,0.0;blue,0.0}, opacity={1.0}, rotate={0}}, xlabel style={at={(ticklabel cs:0.5)}, anchor=near ticklabel, at={{(ticklabel cs:0.5)}}, anchor={near ticklabel}, font={{\fontsize{16 pt}{20.8 pt}\selectfont}}, color={rgb,1:red,0.0;green,0.0;blue,0.0}, draw opacity={1.0}, rotate={0.0}}, xmajorgrids={true}, xmin={0.129999999999999}, xmax={30.87}, xticklabels={{$5$,$10$,$15$,$20$,$25$,$30$}}, xtick={{5.0,10.0,15.0,20.0,25.0,30.0}}, xtick align={inside}, xticklabel style={font={{\fontsize{15 pt}{19.5 pt}\selectfont}}, color={rgb,1:red,0.0;green,0.0;blue,0.0}, draw opacity={1.0}, rotate={0.0}}, x grid style={color={rgb,1:red,0.0;green,0.0;blue,0.0}, draw opacity={0.1}, line width={0.5}, solid}, axis x line*={left}, x axis line style={color={rgb,1:red,0.0;green,0.0;blue,0.0}, draw opacity={1.0}, line width={1}, solid}, scaled y ticks={false}, ylabel={$\|\nabla f\|$}, y tick style={color={rgb,1:red,0.0;green,0.0;blue,0.0}, opacity={1.0}}, y tick label style={color={rgb,1:red,0.0;green,0.0;blue,0.0}, opacity={1.0}, rotate={0}}, ylabel style={at={(ticklabel cs:0.5)}, anchor=near ticklabel, at={{(ticklabel cs:0.5)}}, anchor={near ticklabel}, font={{\fontsize{16 pt}{20.8 pt}\selectfont}}, color={rgb,1:red,0.0;green,0.0;blue,0.0}, draw opacity={1.0}, rotate={0.0}}, ymode={log}, log basis y={10}, ymajorgrids={true}, ymin={5.431720214974174e-10}, ymax={0.602978556751436}, yticklabels={{$10^{-8}$,$10^{-6}$,$10^{-4}$,$10^{-2}$,$10^{-1}$}}, ytick={{1.0e-8,1.0e-6,0.0001,0.01,0.1}}, ytick align={inside}, yticklabel style={font={{\fontsize{15 pt}{19.5 pt}\selectfont}}, color={rgb,1:red,0.0;green,0.0;blue,0.0}, draw opacity={1.0}, rotate={0.0}}, y grid style={color={rgb,1:red,0.0;green,0.0;blue,0.0}, draw opacity={0.1}, line width={0.5}, solid}, axis y line*={left}, y axis line style={color={rgb,1:red,0.0;green,0.0;blue,0.0}, draw opacity={1.0}, line width={1}, solid}, colorbar={false}]
    \addplot[color={rgb,1:red,0.0;green,0.0;blue,0.0}, name path={61}, draw opacity={1.0}, line width={3.5}, solid, mark={*}, mark size={3.0 pt}, mark repeat={1}, mark options={color={rgb,1:red,0.0;green,0.0;blue,0.0}, draw opacity={1.0}, fill={rgb,1:red,0.0;green,0.0;blue,0.0}, fill opacity={1.0}, line width={0.75}, rotate={0}, solid}]
        table[row sep={\\}]
        {
            \\
            1.0  0.3344267332121099  \\
            2.0  0.3262881366662496  \\
            3.0  0.3177935512627032  \\
            4.0  0.3098698369936551  \\
            5.0  0.24457345030175773  \\
            6.0  0.1938841886414492  \\
            7.0  0.07669035742126838  \\
            8.0  0.05569581306465246  \\
            9.0  0.058181269874989186  \\
            10.0  0.033759387400169726  \\
            11.0  0.05076447120689893  \\
            12.0  0.01270179503329759  \\
            13.0  0.011033444928636175  \\
            14.0  0.009863761292885184  \\
            15.0  0.007201265992511347  \\
            16.0  0.09838685186513495  \\
            17.0  0.01869467472215782  \\
            18.0  0.014699547898354462  \\
            19.0  0.0019555393564083456  \\
            20.0  0.007145370786116575  \\
            21.0  0.0011972142155808802  \\
            22.0  0.00012434194794540637  \\
            23.0  2.1201422373000378e-6  \\
            24.0  1.1856510092027596e-8  \\
        }
        ;
    \addlegendentry {$\texttt{ArC}$}
    \addplot[color={rgb,1:red,0.5294;green,0.8078;blue,1.0}, name path={62}, draw opacity={1.0}, line width={3.5}, solid, mark={*}, mark size={3.0 pt}, mark repeat={1}, mark options={color={rgb,1:red,0.0;green,0.0;blue,0.0}, draw opacity={1.0}, fill={rgb,1:red,0.5294;green,0.8078;blue,1.0}, fill opacity={1.0}, line width={0.75}, rotate={0}, solid}]
        table[row sep={\\}]
        {
            \\
            1.0  0.3344267332120956  \\
            2.0  0.3344267332120956  \\
            3.0  0.08733715983850436  \\
            4.0  0.19313031787758814  \\
            5.0  0.19313031787758814  \\
            6.0  0.12748294149820052  \\
            7.0  0.10554167712858556  \\
            8.0  0.10554167712858556  \\
            9.0  0.12090430645526223  \\
            10.0  0.12090430645526223  \\
            11.0  0.12090430645526223  \\
            12.0  0.08559936838718854  \\
            13.0  0.08559936838718854  \\
            14.0  0.10354114792960734  \\
            15.0  0.10354114792960734  \\
            16.0  0.1597258218517446  \\
            17.0  0.11183288496213446  \\
            18.0  0.11183288496213446  \\
            19.0  0.15840938385387474  \\
            20.0  0.10968808079886438  \\
            21.0  0.10968808079886438  \\
            22.0  0.15132549027420783  \\
            23.0  0.06605484261218836  \\
            24.0  0.06246596164968864  \\
            25.0  0.005323311568311702  \\
            26.0  0.0019532474147215006  \\
            27.0  0.00048384731826788376  \\
            28.0  8.096923036412459e-5  \\
            29.0  6.066128066842783e-7  \\
            30.0  1.4995842891192816e-9  \\
        }
        ;
    \addlegendentry {$\texttt{HAR-C} (5)$}
    \addplot[color={rgb,1:red,0.0;green,0.6039;blue,0.8039}, name path={63}, draw opacity={1.0}, line width={3.5}, solid, mark={*}, mark size={3.0 pt}, mark repeat={1}, mark options={color={rgb,1:red,0.0;green,0.0;blue,0.0}, draw opacity={1.0}, fill={rgb,1:red,0.0;green,0.6039;blue,0.8039}, fill opacity={1.0}, line width={0.75}, rotate={0}, solid}]
        table[row sep={\\}]
        {
            \\
            1.0  0.3344267332120956  \\
            2.0  0.3344267332120956  \\
            3.0  0.08733715983850436  \\
            4.0  0.19313031787758814  \\
            5.0  0.19313031787758814  \\
            6.0  0.12748294149820052  \\
            7.0  0.10554167712858556  \\
            8.0  0.01845485680371678  \\
            9.0  0.016354133602211704  \\
            10.0  0.014930107733063578  \\
            11.0  0.013734772179958648  \\
            12.0  0.012675294377961844  \\
            13.0  0.012675294377961844  \\
            14.0  0.021575608589322434  \\
            15.0  0.010410755606692109  \\
            16.0  0.006094532147124914  \\
            17.0  0.003074010398354642  \\
            18.0  0.002448980702260774  \\
            19.0  0.0013589073547665884  \\
            20.0  0.0007593315976756388  \\
            21.0  0.000403622532226242  \\
            22.0  0.0001872594190222108  \\
            23.0  3.2321353082150795e-5  \\
            24.0  1.492544006513098e-7  \\
            25.0  9.793507787026796e-10  \\
        }
        ;
    \addlegendentry {$\texttt{HAR-C} (10)$}
    \addplot[color={rgb,1:red,0.5647;green,0.9333;blue,0.5647}, name path={64}, draw opacity={1.0}, line width={3.5}, solid, mark={*}, mark size={3.0 pt}, mark repeat={1}, mark options={color={rgb,1:red,0.0;green,0.0;blue,0.0}, draw opacity={1.0}, fill={rgb,1:red,0.5647;green,0.9333;blue,0.5647}, fill opacity={1.0}, line width={0.75}, rotate={0}, solid}]
        table[row sep={\\}]
        {
            \\
            1.0  0.3344267332120956  \\
            2.0  0.3344267332120956  \\
            3.0  0.08733715983850436  \\
            4.0  0.19313031787758814  \\
            5.0  0.19313031787758814  \\
            6.0  0.12748294149820052  \\
            7.0  0.10554167712858556  \\
            8.0  0.01845485680371678  \\
            9.0  0.016354133602211704  \\
            10.0  0.014930107733063578  \\
            11.0  0.013734772179958648  \\
            12.0  0.012675294377961844  \\
            13.0  0.011485285081935036  \\
            14.0  0.03073149776386551  \\
            15.0  0.008559226087812198  \\
            16.0  0.014001395882363722  \\
            17.0  0.005524978725236985  \\
            18.0  0.0016367824878035571  \\
            19.0  0.0005707675753154675  \\
            20.0  0.00019418964451129887  \\
            21.0  2.2476587469928626e-5  \\
            22.0  1.6294075395655815e-7  \\
            23.0  1.1226701030184849e-9  \\
        }
        ;
    \addlegendentry {$\texttt{HAR-S} (5)$}
    \addplot[color={rgb,1:red,0.0;green,0.5451;blue,0.0}, name path={65}, draw opacity={1.0}, line width={3.5}, solid, mark={*}, mark size={3.0 pt}, mark repeat={1}, mark options={color={rgb,1:red,0.0;green,0.0;blue,0.0}, draw opacity={1.0}, fill={rgb,1:red,0.0;green,0.5451;blue,0.0}, fill opacity={1.0}, line width={0.75}, rotate={0}, solid}]
        table[row sep={\\}]
        {
            \\
            1.0  0.3344267332120956  \\
            2.0  0.3344267332120956  \\
            3.0  0.08733715983850436  \\
            4.0  0.19313031787758814  \\
            5.0  0.19313031787758814  \\
            6.0  0.12748294149820052  \\
            7.0  0.10554167712858556  \\
            8.0  0.01845485680371678  \\
            9.0  0.016354133602211704  \\
            10.0  0.014930107733063578  \\
            11.0  0.013734772179958648  \\
            12.0  0.012675294377961844  \\
            13.0  0.011724773921631897  \\
            14.0  0.010870960182096446  \\
            15.0  0.010109625848075987  \\
            16.0  0.00944311414899865  \\
            17.0  0.008843062622837326  \\
            18.0  0.008124482876635569  \\
            19.0  0.010164621847826495  \\
            20.0  0.007650688866348523  \\
            21.0  0.003250329774905433  \\
            22.0  0.0028283159640444433  \\
            23.0  0.0010151831238657321  \\
            24.0  0.0003690845956281926  \\
            25.0  0.00013109608025981318  \\
            26.0  2.332221163288946e-5  \\
            27.0  9.160723222242853e-7  \\
            28.0  7.90192385219567e-9  \\
        }
        ;
    \addlegendentry {$\texttt{HAR-S} (10)$}
\end{axis}
\end{tikzpicture}

%% file: figs/e-logistic-w4a-k.tex
% Recommended preamble:
% \usetikzlibrary{arrows.meta}
% \usetikzlibrary{backgrounds}
% \usepgfplotslibrary{patchplots}
% \usepgfplotslibrary{fillbetween}
% \pgfplotsset{%
%     layers/standard/.define layer set={%
%         background,axis background,axis grid,axis ticks,axis lines,axis tick labels,pre main,main,axis descriptions,axis foreground%
%     }{
%         grid style={/pgfplots/on layer=axis grid},%
%         tick style={/pgfplots/on layer=axis ticks},%
%         axis line style={/pgfplots/on layer=axis lines},%
%         label style={/pgfplots/on layer=axis descriptions},%
%         legend style={/pgfplots/on layer=axis descriptions},%
%         title style={/pgfplots/on layer=axis descriptions},%
%         colorbar style={/pgfplots/on layer=axis descriptions},%
%         ticklabel style={/pgfplots/on layer=axis tick labels},%
%         axis background@ style={/pgfplots/on layer=axis background},%
%         3d box foreground style={/pgfplots/on layer=axis foreground},%
%     },
% }

\begin{tikzpicture}[/tikz/background rectangle/.style={fill={rgb,1:red,1.0;green,1.0;blue,1.0}, fill opacity={1.0}, draw opacity={1.0}}, show background rectangle]
\begin{axis}[point meta max={nan}, point meta min={nan}, legend cell align={left}, legend columns={1}, title={Logistic Regression $\texttt{name}:=\texttt{w4a}$, $n:=$300, $N:=$7366}, title style={at={{(0.5,1)}}, anchor={south}, font={{\fontsize{22 pt}{28.6 pt}\selectfont}}, color={rgb,1:red,0.0;green,0.0;blue,0.0}, draw opacity={1.0}, rotate={0.0}, align={center}}, legend style={color={rgb,1:red,0.0;green,0.0;blue,0.0}, draw opacity={1.0}, line width={1}, solid, fill={rgb,1:red,1.0;green,1.0;blue,1.0}, fill opacity={1.0}, text opacity={1.0}, font={{\fontsize{16 pt}{20.8 pt}\selectfont}}, text={rgb,1:red,0.0;green,0.0;blue,0.0}, cells={anchor={center}}, at={(1.02, 1)}, anchor={north west}}, axis background/.style={fill={rgb,1:red,1.0;green,1.0;blue,1.0}, opacity={1.0}}, anchor={north west}, xshift={1.0mm}, yshift={-1.0mm}, width={145.4mm}, height={125.0mm}, scaled x ticks={false}, xlabel={$\textrm{Iterations}$}, x tick style={color={rgb,1:red,0.0;green,0.0;blue,0.0}, opacity={1.0}}, x tick label style={color={rgb,1:red,0.0;green,0.0;blue,0.0}, opacity={1.0}, rotate={0}}, xlabel style={at={(ticklabel cs:0.5)}, anchor=near ticklabel, at={{(ticklabel cs:0.5)}}, anchor={near ticklabel}, font={{\fontsize{16 pt}{20.8 pt}\selectfont}}, color={rgb,1:red,0.0;green,0.0;blue,0.0}, draw opacity={1.0}, rotate={0.0}}, xmajorgrids={true}, xmin={0.33999999999999986}, xmax={23.66}, xticklabels={{$5$,$10$,$15$,$20$}}, xtick={{5.0,10.0,15.0,20.0}}, xtick align={inside}, xticklabel style={font={{\fontsize{15 pt}{19.5 pt}\selectfont}}, color={rgb,1:red,0.0;green,0.0;blue,0.0}, draw opacity={1.0}, rotate={0.0}}, x grid style={color={rgb,1:red,0.0;green,0.0;blue,0.0}, draw opacity={0.1}, line width={0.5}, solid}, axis x line*={left}, x axis line style={color={rgb,1:red,0.0;green,0.0;blue,0.0}, draw opacity={1.0}, line width={1}, solid}, scaled y ticks={false}, ylabel={$\|\nabla f\|$}, y tick style={color={rgb,1:red,0.0;green,0.0;blue,0.0}, opacity={1.0}}, y tick label style={color={rgb,1:red,0.0;green,0.0;blue,0.0}, opacity={1.0}, rotate={0}}, ylabel style={at={(ticklabel cs:0.5)}, anchor=near ticklabel, at={{(ticklabel cs:0.5)}}, anchor={near ticklabel}, font={{\fontsize{16 pt}{20.8 pt}\selectfont}}, color={rgb,1:red,0.0;green,0.0;blue,0.0}, draw opacity={1.0}, rotate={0.0}}, ymode={log}, log basis y={10}, ymajorgrids={true}, ymin={2.551112981425958e-11}, ymax={0.3985736402519928}, yticklabels={{$10^{-10}$,$10^{-8}$,$10^{-6}$,$10^{-4}$,$10^{-2}$,$10^{-1}$}}, ytick={{1.0e-10,1.0e-8,1.0e-6,0.0001,0.01,0.1}}, ytick align={inside}, yticklabel style={font={{\fontsize{15 pt}{19.5 pt}\selectfont}}, color={rgb,1:red,0.0;green,0.0;blue,0.0}, draw opacity={1.0}, rotate={0.0}}, y grid style={color={rgb,1:red,0.0;green,0.0;blue,0.0}, draw opacity={0.1}, line width={0.5}, solid}, axis y line*={left}, y axis line style={color={rgb,1:red,0.0;green,0.0;blue,0.0}, draw opacity={1.0}, line width={1}, solid}, colorbar={false}]
    \addplot[color={rgb,1:red,0.0;green,0.0;blue,0.0}, name path={106}, draw opacity={1.0}, line width={3.5}, solid, mark={*}, mark size={3.0 pt}, mark repeat={1}, mark options={color={rgb,1:red,0.0;green,0.0;blue,0.0}, draw opacity={1.0}, fill={rgb,1:red,0.0;green,0.0;blue,0.0}, fill opacity={1.0}, line width={0.75}, rotate={0}, solid}]
        table[row sep={\\}]
        {
            \\
            1.0  0.2051187670292653  \\
            2.0  0.20427175048627574  \\
            3.0  0.20342435448446772  \\
            4.0  0.1952226742788686  \\
            5.0  0.18697581345525832  \\
            6.0  0.17833161799433608  \\
            7.0  0.11914867443279986  \\
            8.0  0.0751933881265646  \\
            9.0  0.02595227508909039  \\
            10.0  0.01827448631083919  \\
            11.0  0.01501362492424938  \\
            12.0  0.06476048604730344  \\
            13.0  0.01422763864788513  \\
            14.0  0.00434324050829759  \\
            15.0  0.002853080599431805  \\
            16.0  0.019036280701978717  \\
            17.0  0.011232588599122156  \\
            18.0  0.004202300544461568  \\
            19.0  0.0012591112195512984  \\
            20.0  0.00025556703829227533  \\
            21.0  1.872037679425773e-5  \\
            22.0  1.382802679933507e-7  \\
            23.0  1.2211493308627298e-8  \\
        }
        ;
    \addlegendentry {$\texttt{ArC}$}
    \addplot[color={rgb,1:red,0.5294;green,0.8078;blue,1.0}, name path={107}, draw opacity={1.0}, line width={3.5}, solid, mark={*}, mark size={3.0 pt}, mark repeat={1}, mark options={color={rgb,1:red,0.0;green,0.0;blue,0.0}, draw opacity={1.0}, fill={rgb,1:red,0.5294;green,0.8078;blue,1.0}, fill opacity={1.0}, line width={0.75}, rotate={0}, solid}]
        table[row sep={\\}]
        {
            \\
            1.0  0.2051187670292654  \\
            2.0  0.2051187670292654  \\
            3.0  0.02242371808070856  \\
            4.0  0.017858014045382613  \\
            5.0  0.015197012430325835  \\
            6.0  0.0129992900328267  \\
            7.0  0.010952835473755529  \\
            8.0  0.008909623931734324  \\
            9.0  0.007081185561879985  \\
            10.0  0.005450330776094125  \\
            11.0  0.00408011933544901  \\
            12.0  0.002849217996254034  \\
            13.0  0.0019510368486670205  \\
            14.0  0.0011239398930271126  \\
            15.0  0.0005623375075351146  \\
            16.0  0.0002172531719343428  \\
            17.0  5.669299043288747e-5  \\
            18.0  6.162693454848276e-6  \\
            19.0  9.617754779302462e-8  \\
            20.0  4.957159222568772e-11  \\
        }
        ;
    \addlegendentry {$\texttt{HAR-C} (5)$}
    \addplot[color={rgb,1:red,0.0;green,0.6039;blue,0.8039}, name path={108}, draw opacity={1.0}, line width={3.5}, solid, mark={*}, mark size={3.0 pt}, mark repeat={1}, mark options={color={rgb,1:red,0.0;green,0.0;blue,0.0}, draw opacity={1.0}, fill={rgb,1:red,0.0;green,0.6039;blue,0.8039}, fill opacity={1.0}, line width={0.75}, rotate={0}, solid}]
        table[row sep={\\}]
        {
            \\
            1.0  0.2051187670292654  \\
            2.0  0.2051187670292654  \\
            3.0  0.02242371808070856  \\
            4.0  0.017858014045382613  \\
            5.0  0.015197012430325835  \\
            6.0  0.0129992900328267  \\
            7.0  0.010952835473755529  \\
            8.0  0.00923572620979711  \\
            9.0  0.007683776184781019  \\
            10.0  0.006233434398493313  \\
            11.0  0.0049162797156128745  \\
            12.0  0.003783694189401284  \\
            13.0  0.00267860735341797  \\
            14.0  0.0018250310627739444  \\
            15.0  0.00100785969500058  \\
            16.0  0.00048304698474667493  \\
            17.0  0.00017586691420607656  \\
            18.0  4.1036670880317315e-5  \\
            19.0  3.4742038209091497e-6  \\
            20.0  3.0088548693480565e-8  \\
        }
        ;
    \addlegendentry {$\texttt{HAR-C} (10)$}
    \addplot[color={rgb,1:red,0.5647;green,0.9333;blue,0.5647}, name path={109}, draw opacity={1.0}, line width={3.5}, solid, mark={*}, mark size={3.0 pt}, mark repeat={1}, mark options={color={rgb,1:red,0.0;green,0.0;blue,0.0}, draw opacity={1.0}, fill={rgb,1:red,0.5647;green,0.9333;blue,0.5647}, fill opacity={1.0}, line width={0.75}, rotate={0}, solid}]
        table[row sep={\\}]
        {
            \\
            1.0  0.2051187670292654  \\
            2.0  0.2051187670292654  \\
            3.0  0.02242371808070856  \\
            4.0  0.017858014045382613  \\
            5.0  0.015197012430325835  \\
            6.0  0.0129992900328267  \\
            7.0  0.010952835473755529  \\
            8.0  0.00923572620979711  \\
            9.0  0.016243049373074976  \\
            10.0  0.016243049373074976  \\
            11.0  0.008096029447401815  \\
            12.0  0.010513817531338574  \\
            13.0  0.0025096673514182484  \\
            14.0  0.0007053529483415285  \\
            15.0  0.00011843804067375144  \\
            16.0  7.441580385285255e-6  \\
            17.0  9.715217360460524e-8  \\
            18.0  6.347218099391651e-11  \\
        }
        ;
    \addlegendentry {$\texttt{HAR-S} (5)$}
    \addplot[color={rgb,1:red,0.0;green,0.5451;blue,0.0}, name path={110}, draw opacity={1.0}, line width={3.5}, solid, mark={*}, mark size={3.0 pt}, mark repeat={1}, mark options={color={rgb,1:red,0.0;green,0.0;blue,0.0}, draw opacity={1.0}, fill={rgb,1:red,0.0;green,0.5451;blue,0.0}, fill opacity={1.0}, line width={0.75}, rotate={0}, solid}]
        table[row sep={\\}]
        {
            \\
            1.0  0.2051187670292654  \\
            2.0  0.2051187670292654  \\
            3.0  0.02242371808070856  \\
            4.0  0.017858014045382613  \\
            5.0  0.015197012430325835  \\
            6.0  0.0129992900328267  \\
            7.0  0.010952835473755529  \\
            8.0  0.00923572620979711  \\
            9.0  0.007683776184781019  \\
            10.0  0.006233434398493313  \\
            11.0  0.0049162797156128745  \\
            12.0  0.003783694189401284  \\
            13.0  0.0028147838105399194  \\
            14.0  0.021350572069128368  \\
            15.0  0.006030794544138465  \\
            16.0  0.002759983354769662  \\
            17.0  0.0006393632319899959  \\
            18.0  0.00011866759387242786  \\
            19.0  9.615825418880014e-6  \\
            20.0  1.7729436917563852e-7  \\
            21.0  1.8558773498426337e-10  \\
        }
        ;
    \addlegendentry {$\texttt{HAR-S} (10)$}
\end{axis}
\end{tikzpicture}

%% file: figs/e-logistic-w8a-k.tex
% Recommended preamble:
% \usetikzlibrary{arrows.meta}
% \usetikzlibrary{backgrounds}
% \usepgfplotslibrary{patchplots}
% \usepgfplotslibrary{fillbetween}
% \pgfplotsset{%
%     layers/standard/.define layer set={%
%         background,axis background,axis grid,axis ticks,axis lines,axis tick labels,pre main,main,axis descriptions,axis foreground%
%     }{
%         grid style={/pgfplots/on layer=axis grid},%
%         tick style={/pgfplots/on layer=axis ticks},%
%         axis line style={/pgfplots/on layer=axis lines},%
%         label style={/pgfplots/on layer=axis descriptions},%
%         legend style={/pgfplots/on layer=axis descriptions},%
%         title style={/pgfplots/on layer=axis descriptions},%
%         colorbar style={/pgfplots/on layer=axis descriptions},%
%         ticklabel style={/pgfplots/on layer=axis tick labels},%
%         axis background@ style={/pgfplots/on layer=axis background},%
%         3d box foreground style={/pgfplots/on layer=axis foreground},%
%     },
% }

\begin{tikzpicture}[/tikz/background rectangle/.style={fill={rgb,1:red,1.0;green,1.0;blue,1.0}, fill opacity={1.0}, draw opacity={1.0}}, show background rectangle]
\begin{axis}[point meta max={nan}, point meta min={nan}, legend cell align={left}, legend columns={1}, title={Logistic Regression $\texttt{name}:=\texttt{w8a}$, $n:=$300, $N:=$49749}, title style={at={{(0.5,1)}}, anchor={south}, font={{\fontsize{22 pt}{28.6 pt}\selectfont}}, color={rgb,1:red,0.0;green,0.0;blue,0.0}, draw opacity={1.0}, rotate={0.0}, align={center}}, legend style={color={rgb,1:red,0.0;green,0.0;blue,0.0}, draw opacity={1.0}, line width={1}, solid, fill={rgb,1:red,1.0;green,1.0;blue,1.0}, fill opacity={1.0}, text opacity={1.0}, font={{\fontsize{16 pt}{20.8 pt}\selectfont}}, text={rgb,1:red,0.0;green,0.0;blue,0.0}, cells={anchor={center}}, at={(1.02, 1)}, anchor={north west}}, axis background/.style={fill={rgb,1:red,1.0;green,1.0;blue,1.0}, opacity={1.0}}, anchor={north west}, xshift={1.0mm}, yshift={-1.0mm}, width={145.4mm}, height={125.0mm}, scaled x ticks={false}, xlabel={$\textrm{Iterations}$}, x tick style={color={rgb,1:red,0.0;green,0.0;blue,0.0}, opacity={1.0}}, x tick label style={color={rgb,1:red,0.0;green,0.0;blue,0.0}, opacity={1.0}, rotate={0}}, xlabel style={at={(ticklabel cs:0.5)}, anchor=near ticklabel, at={{(ticklabel cs:0.5)}}, anchor={near ticklabel}, font={{\fontsize{16 pt}{20.8 pt}\selectfont}}, color={rgb,1:red,0.0;green,0.0;blue,0.0}, draw opacity={1.0}, rotate={0.0}}, xmajorgrids={true}, xmin={0.25}, xmax={26.75}, xticklabels={{$5$,$10$,$15$,$20$,$25$}}, xtick={{5.0,10.0,15.0,20.0,25.0}}, xtick align={inside}, xticklabel style={font={{\fontsize{15 pt}{19.5 pt}\selectfont}}, color={rgb,1:red,0.0;green,0.0;blue,0.0}, draw opacity={1.0}, rotate={0.0}}, x grid style={color={rgb,1:red,0.0;green,0.0;blue,0.0}, draw opacity={0.1}, line width={0.5}, solid}, axis x line*={left}, x axis line style={color={rgb,1:red,0.0;green,0.0;blue,0.0}, draw opacity={1.0}, line width={1}, solid}, scaled y ticks={false}, ylabel={$\|\nabla f\|$}, y tick style={color={rgb,1:red,0.0;green,0.0;blue,0.0}, opacity={1.0}}, y tick label style={color={rgb,1:red,0.0;green,0.0;blue,0.0}, opacity={1.0}, rotate={0}}, ylabel style={at={(ticklabel cs:0.5)}, anchor=near ticklabel, at={{(ticklabel cs:0.5)}}, anchor={near ticklabel}, font={{\fontsize{16 pt}{20.8 pt}\selectfont}}, color={rgb,1:red,0.0;green,0.0;blue,0.0}, draw opacity={1.0}, rotate={0.0}}, ymode={log}, log basis y={10}, ymajorgrids={true}, ymin={3.8628593137372757e-10}, ymax={0.36508458379653796}, yticklabels={{$10^{-8}$,$10^{-6}$,$10^{-4}$,$10^{-2}$,$10^{-1}$}}, ytick={{1.0e-8,1.0e-6,0.0001,0.01,0.1}}, ytick align={inside}, yticklabel style={font={{\fontsize{15 pt}{19.5 pt}\selectfont}}, color={rgb,1:red,0.0;green,0.0;blue,0.0}, draw opacity={1.0}, rotate={0.0}}, y grid style={color={rgb,1:red,0.0;green,0.0;blue,0.0}, draw opacity={0.1}, line width={0.5}, solid}, axis y line*={left}, y axis line style={color={rgb,1:red,0.0;green,0.0;blue,0.0}, draw opacity={1.0}, line width={1}, solid}, colorbar={false}]
    \addplot[color={rgb,1:red,0.0;green,0.0;blue,0.0}, name path={166}, draw opacity={1.0}, line width={3.5}, solid, mark={*}, mark size={3.0 pt}, mark repeat={1}, mark options={color={rgb,1:red,0.0;green,0.0;blue,0.0}, draw opacity={1.0}, fill={rgb,1:red,0.0;green,0.0;blue,0.0}, fill opacity={1.0}, line width={0.75}, rotate={0}, solid}]
        table[row sep={\\}]
        {
            \\
            1.0  0.20340908466526636  \\
            2.0  0.20257441875312146  \\
            3.0  0.2017796484441946  \\
            4.0  0.19342742072600352  \\
            5.0  0.18550077795061534  \\
            6.0  0.17711846811593018  \\
            7.0  0.11551156220423286  \\
            8.0  0.07551283593163395  \\
            9.0  0.025039610573050694  \\
            10.0  0.014840971188930518  \\
            11.0  0.00674511624102459  \\
            12.0  0.004374724633388834  \\
            13.0  0.0028580332367479262  \\
            14.0  0.00209132052447465  \\
            15.0  0.0019043491003423888  \\
            16.0  0.001489832819121755  \\
            17.0  0.0016017632664127566  \\
            18.0  0.0004857526163630629  \\
            19.0  8.827127915625591e-5  \\
            20.0  8.747366521331444e-6  \\
            21.0  2.8561098969349225e-7  \\
            22.0  1.151536478337084e-8  \\
        }
        ;
    \addlegendentry {$\texttt{ArC}$}
    \addplot[color={rgb,1:red,0.5294;green,0.8078;blue,1.0}, name path={167}, draw opacity={1.0}, line width={3.5}, solid, mark={*}, mark size={3.0 pt}, mark repeat={1}, mark options={color={rgb,1:red,0.0;green,0.0;blue,0.0}, draw opacity={1.0}, fill={rgb,1:red,0.5294;green,0.8078;blue,1.0}, fill opacity={1.0}, line width={0.75}, rotate={0}, solid}]
        table[row sep={\\}]
        {
            \\
            1.0  0.20340908466526803  \\
            2.0  0.20340908466526803  \\
            3.0  0.017340505244313444  \\
            4.0  0.012070076305042875  \\
            5.0  0.010232072302132456  \\
            6.0  0.00881150202574867  \\
            7.0  0.0076104204801954334  \\
            8.0  0.006586014207649286  \\
            9.0  0.005643854700368755  \\
            10.0  0.004782371482061252  \\
            11.0  0.004012461943804009  \\
            12.0  0.003328565970040239  \\
            13.0  0.0026977721678370166  \\
            14.0  0.0021362743398144085  \\
            15.0  0.0016199909847201314  \\
            16.0  0.00119366443765023  \\
            17.0  0.0008390715202045491  \\
            18.0  0.0005655687092379867  \\
            19.0  0.00035397671322959186  \\
            20.0  0.00019936867388431403  \\
            21.0  9.05262136048425e-5  \\
            22.0  3.09925814323898e-5  \\
            23.0  6.269170506177677e-6  \\
            24.0  3.7337128272845067e-7  \\
            25.0  1.1301725195494032e-9  \\
        }
        ;
    \addlegendentry {$\texttt{HAR-C} (5)$}
    \addplot[color={rgb,1:red,0.0;green,0.6039;blue,0.8039}, name path={168}, draw opacity={1.0}, line width={3.5}, solid, mark={*}, mark size={3.0 pt}, mark repeat={1}, mark options={color={rgb,1:red,0.0;green,0.0;blue,0.0}, draw opacity={1.0}, fill={rgb,1:red,0.0;green,0.6039;blue,0.8039}, fill opacity={1.0}, line width={0.75}, rotate={0}, solid}]
        table[row sep={\\}]
        {
            \\
            1.0  0.20340908466526803  \\
            2.0  0.20340908466526803  \\
            3.0  0.017340505244313444  \\
            4.0  0.012070076305042875  \\
            5.0  0.010232072302132456  \\
            6.0  0.00881150202574867  \\
            7.0  0.0076104204801954334  \\
            8.0  0.006757182334642275  \\
            9.0  0.0059609861588132565  \\
            10.0  0.005228837850763585  \\
            11.0  0.00454346438721511  \\
            12.0  0.003920328015003756  \\
            13.0  0.003246103822774484  \\
            14.0  0.0026287134172969934  \\
            15.0  0.002073499968395028  \\
            16.0  0.0015637718940231676  \\
            17.0  0.0011484745035919063  \\
            18.0  0.0008027779890646496  \\
            19.0  0.0005372049934788684  \\
            20.0  0.0003328750684552238  \\
            21.0  0.0001844033763108874  \\
            22.0  8.1057252907809e-5  \\
            23.0  2.6619101044259078e-5  \\
            24.0  4.890823387958522e-6  \\
            25.0  2.2990260524273324e-7  \\
            26.0  6.964428285697176e-10  \\
        }
        ;
    \addlegendentry {$\texttt{HAR-C} (10)$}
    \addplot[color={rgb,1:red,0.5647;green,0.9333;blue,0.5647}, name path={169}, draw opacity={1.0}, line width={3.5}, solid, mark={*}, mark size={3.0 pt}, mark repeat={1}, mark options={color={rgb,1:red,0.0;green,0.0;blue,0.0}, draw opacity={1.0}, fill={rgb,1:red,0.5647;green,0.9333;blue,0.5647}, fill opacity={1.0}, line width={0.75}, rotate={0}, solid}]
        table[row sep={\\}]
        {
            \\
            1.0  0.20340908466526803  \\
            2.0  0.20340908466526803  \\
            3.0  0.017340505244313444  \\
            4.0  0.012070076305042875  \\
            5.0  0.010232072302132456  \\
            6.0  0.00881150202574867  \\
            7.0  0.006037734946705061  \\
            8.0  0.0108407354832087  \\
            9.0  0.02600443666618287  \\
            10.0  0.03696872057240327  \\
            11.0  0.012209935815516744  \\
            12.0  0.004679464911797714  \\
            13.0  0.0017599322994573563  \\
            14.0  0.00040122922023359586  \\
            15.0  0.0002979269397092944  \\
            16.0  9.431996204758694e-5  \\
            17.0  2.6663006403559305e-5  \\
            18.0  4.636962291867871e-6  \\
            19.0  2.2507105563400845e-7  \\
            20.0  6.933173054394812e-10  \\
        }
        ;
    \addlegendentry {$\texttt{HAR-S} (5)$}
    \addplot[color={rgb,1:red,0.0;green,0.5451;blue,0.0}, name path={170}, draw opacity={1.0}, line width={3.5}, solid, mark={*}, mark size={3.0 pt}, mark repeat={1}, mark options={color={rgb,1:red,0.0;green,0.0;blue,0.0}, draw opacity={1.0}, fill={rgb,1:red,0.0;green,0.5451;blue,0.0}, fill opacity={1.0}, line width={0.75}, rotate={0}, solid}]
        table[row sep={\\}]
        {
            \\
            1.0  0.20340908466526803  \\
            2.0  0.20340908466526803  \\
            3.0  0.017340505244313444  \\
            4.0  0.012070076305042875  \\
            5.0  0.010232072302132456  \\
            6.0  0.00881150202574867  \\
            7.0  0.0076104204801954334  \\
            8.0  0.006757182334642275  \\
            9.0  0.0059609861588132565  \\
            10.0  0.005228837850763585  \\
            11.0  0.00454346438721511  \\
            12.0  0.003920328015003756  \\
            13.0  0.003348488583901012  \\
            14.0  0.0027198493643841707  \\
            15.0  0.003774720635886934  \\
            16.0  0.0011571934175293626  \\
            17.0  0.00039408974937074046  \\
            18.0  8.204640814628381e-5  \\
            19.0  1.036139404569737e-5  \\
            20.0  5.676649914879343e-7  \\
            21.0  3.4183981663419398e-9  \\
        }
        ;
    \addlegendentry {$\texttt{HAR-S} (10)$}
\end{axis}
\end{tikzpicture}

%% file: figs/e-logistic-mushroom-k.tex
% Recommended preamble:
% \usetikzlibrary{arrows.meta}
% \usetikzlibrary{backgrounds}
% \usepgfplotslibrary{patchplots}
% \usepgfplotslibrary{fillbetween}
% \pgfplotsset{%
%     layers/standard/.define layer set={%
%         background,axis background,axis grid,axis ticks,axis lines,axis tick labels,pre main,main,axis descriptions,axis foreground%
%     }{
%         grid style={/pgfplots/on layer=axis grid},%
%         tick style={/pgfplots/on layer=axis ticks},%
%         axis line style={/pgfplots/on layer=axis lines},%
%         label style={/pgfplots/on layer=axis descriptions},%
%         legend style={/pgfplots/on layer=axis descriptions},%
%         title style={/pgfplots/on layer=axis descriptions},%
%         colorbar style={/pgfplots/on layer=axis descriptions},%
%         ticklabel style={/pgfplots/on layer=axis tick labels},%
%         axis background@ style={/pgfplots/on layer=axis background},%
%         3d box foreground style={/pgfplots/on layer=axis foreground},%
%     },
% }

\begin{tikzpicture}[/tikz/background rectangle/.style={fill={rgb,1:red,1.0;green,1.0;blue,1.0}, fill opacity={1.0}, draw opacity={1.0}}, show background rectangle]
\begin{axis}[point meta max={nan}, point meta min={nan}, legend cell align={left}, legend columns={1}, title={Logistic Regression $\texttt{name}:=\texttt{mushroom}$, $n:=$112, $N:=$8124}, title style={at={{(0.5,1)}}, anchor={south}, font={{\fontsize{22 pt}{28.6 pt}\selectfont}}, color={rgb,1:red,0.0;green,0.0;blue,0.0}, draw opacity={1.0}, rotate={0.0}, align={center}}, legend style={color={rgb,1:red,0.0;green,0.0;blue,0.0}, draw opacity={1.0}, line width={1}, solid, fill={rgb,1:red,1.0;green,1.0;blue,1.0}, fill opacity={1.0}, text opacity={1.0}, font={{\fontsize{16 pt}{20.8 pt}\selectfont}}, text={rgb,1:red,0.0;green,0.0;blue,0.0}, cells={anchor={center}}, at={(1.02, 1)}, anchor={north west}}, axis background/.style={fill={rgb,1:red,1.0;green,1.0;blue,1.0}, opacity={1.0}}, anchor={north west}, xshift={1.0mm}, yshift={-1.0mm}, width={145.4mm}, height={125.0mm}, scaled x ticks={false}, xlabel={$\textrm{Iterations}$}, x tick style={color={rgb,1:red,0.0;green,0.0;blue,0.0}, opacity={1.0}}, x tick label style={color={rgb,1:red,0.0;green,0.0;blue,0.0}, opacity={1.0}, rotate={0}}, xlabel style={at={(ticklabel cs:0.5)}, anchor=near ticklabel, at={{(ticklabel cs:0.5)}}, anchor={near ticklabel}, font={{\fontsize{16 pt}{20.8 pt}\selectfont}}, color={rgb,1:red,0.0;green,0.0;blue,0.0}, draw opacity={1.0}, rotate={0.0}}, xmajorgrids={true}, xmin={0.33999999999999986}, xmax={23.66}, xticklabels={{$5$,$10$,$15$,$20$}}, xtick={{5.0,10.0,15.0,20.0}}, xtick align={inside}, xticklabel style={font={{\fontsize{15 pt}{19.5 pt}\selectfont}}, color={rgb,1:red,0.0;green,0.0;blue,0.0}, draw opacity={1.0}, rotate={0.0}}, x grid style={color={rgb,1:red,0.0;green,0.0;blue,0.0}, draw opacity={0.1}, line width={0.5}, solid}, axis x line*={left}, x axis line style={color={rgb,1:red,0.0;green,0.0;blue,0.0}, draw opacity={1.0}, line width={1}, solid}, scaled y ticks={false}, ylabel={$\|\nabla f\|$}, y tick style={color={rgb,1:red,0.0;green,0.0;blue,0.0}, opacity={1.0}}, y tick label style={color={rgb,1:red,0.0;green,0.0;blue,0.0}, opacity={1.0}, rotate={0}}, ylabel style={at={(ticklabel cs:0.5)}, anchor=near ticklabel, at={{(ticklabel cs:0.5)}}, anchor={near ticklabel}, font={{\fontsize{16 pt}{20.8 pt}\selectfont}}, color={rgb,1:red,0.0;green,0.0;blue,0.0}, draw opacity={1.0}, rotate={0.0}}, ymode={log}, log basis y={10}, ymajorgrids={true}, ymin={3.557443132734769e-11}, ymax={1.240948897367073}, yticklabels={{$10^{-10}$,$10^{-8}$,$10^{-6}$,$10^{-4}$,$10^{-2}$,$10^{-1}$,$10^{0}$}}, ytick={{1.0e-10,1.0e-8,1.0e-6,0.0001,0.01,0.1,1.0}}, ytick align={inside}, yticklabel style={font={{\fontsize{15 pt}{19.5 pt}\selectfont}}, color={rgb,1:red,0.0;green,0.0;blue,0.0}, draw opacity={1.0}, rotate={0.0}}, y grid style={color={rgb,1:red,0.0;green,0.0;blue,0.0}, draw opacity={0.1}, line width={0.5}, solid}, axis y line*={left}, y axis line style={color={rgb,1:red,0.0;green,0.0;blue,0.0}, draw opacity={1.0}, line width={1}, solid}, colorbar={false}]
    \addplot[color={rgb,1:red,0.0;green,0.0;blue,0.0}, name path={91}, draw opacity={1.0}, line width={3.5}, solid, mark={*}, mark size={3.0 pt}, mark repeat={1}, mark options={color={rgb,1:red,0.0;green,0.0;blue,0.0}, draw opacity={1.0}, fill={rgb,1:red,0.0;green,0.0;blue,0.0}, fill opacity={1.0}, line width={0.75}, rotate={0}, solid}]
        table[row sep={\\}]
        {
            \\
            1.0  0.6242779203114507  \\
            2.0  0.5982231076135556  \\
            3.0  0.5710463741126743  \\
            4.0  0.5437732894949991  \\
            5.0  0.3601068321968246  \\
            6.0  0.2499615644771413  \\
            7.0  0.07567566383946246  \\
            8.0  0.0172216756849562  \\
            9.0  0.008221427077451928  \\
            10.0  0.007305861372141136  \\
            11.0  0.0036243139696144674  \\
            12.0  0.0012699705659074082  \\
            13.0  0.0005673203338128687  \\
            14.0  0.0001960889891033091  \\
            15.0  5.515819316167755e-5  \\
            16.0  8.659679779060088e-6  \\
            17.0  3.290386217802496e-7  \\
            18.0  1.012238224776408e-8  \\
        }
        ;
    \addlegendentry {$\texttt{ArC}$}
    \addplot[color={rgb,1:red,0.5294;green,0.8078;blue,1.0}, name path={92}, draw opacity={1.0}, line width={3.5}, solid, mark={*}, mark size={3.0 pt}, mark repeat={1}, mark options={color={rgb,1:red,0.0;green,0.0;blue,0.0}, draw opacity={1.0}, fill={rgb,1:red,0.5294;green,0.8078;blue,1.0}, fill opacity={1.0}, line width={0.75}, rotate={0}, solid}]
        table[row sep={\\}]
        {
            \\
            1.0  0.6242779203114784  \\
            2.0  0.6242779203114784  \\
            3.0  0.009189261764723162  \\
            4.0  0.008240879836451815  \\
            5.0  0.007345490393495181  \\
            6.0  0.00650301009314599  \\
            7.0  0.005713311054911668  \\
            8.0  0.004442457959108869  \\
            9.0  0.003340181797439136  \\
            10.0  0.0024050574958614157  \\
            11.0  0.0016340564721114103  \\
            12.0  0.0010233061772842742  \\
            13.0  0.0005671906569142889  \\
            14.0  0.0002593000115968517  \\
            15.0  8.98173722916233e-5  \\
            16.0  1.5011704358824298e-5  \\
            17.0  5.57694381191392e-7  \\
            18.0  1.0452607954965322e-9  \\
        }
        ;
    \addlegendentry {$\texttt{HAR-C} (5)$}
    \addplot[color={rgb,1:red,0.0;green,0.6039;blue,0.8039}, name path={93}, draw opacity={1.0}, line width={3.5}, solid, mark={*}, mark size={3.0 pt}, mark repeat={1}, mark options={color={rgb,1:red,0.0;green,0.0;blue,0.0}, draw opacity={1.0}, fill={rgb,1:red,0.0;green,0.6039;blue,0.8039}, fill opacity={1.0}, line width={0.75}, rotate={0}, solid}]
        table[row sep={\\}]
        {
            \\
            1.0  0.6242779203114784  \\
            2.0  0.6242779203114784  \\
            3.0  0.009189261764723162  \\
            4.0  0.008240879836451815  \\
            5.0  0.007345490393495181  \\
            6.0  0.00650301009314599  \\
            7.0  0.005713311054911668  \\
            8.0  0.004976318050307804  \\
            9.0  0.0042918433379935264  \\
            10.0  0.0036598744543256886  \\
            11.0  0.0030799972974173404  \\
            12.0  0.0025521638779829386  \\
            13.0  0.001753487311287905  \\
            14.0  0.0011157461557504127  \\
            15.0  0.0006338906927000149  \\
            16.0  0.00030050477075634053  \\
            17.0  0.00011326218892417641  \\
            18.0  2.2106747658849025e-5  \\
            19.0  1.1668228632426866e-6  \\
            20.0  4.096521160108457e-9  \\
        }
        ;
    \addlegendentry {$\texttt{HAR-C} (10)$}
    \addplot[color={rgb,1:red,0.5647;green,0.9333;blue,0.5647}, name path={94}, draw opacity={1.0}, line width={3.5}, solid, mark={*}, mark size={3.0 pt}, mark repeat={1}, mark options={color={rgb,1:red,0.0;green,0.0;blue,0.0}, draw opacity={1.0}, fill={rgb,1:red,0.5647;green,0.9333;blue,0.5647}, fill opacity={1.0}, line width={0.75}, rotate={0}, solid}]
        table[row sep={\\}]
        {
            \\
            1.0  0.6242779203114784  \\
            2.0  0.6242779203114784  \\
            3.0  0.009189261764723162  \\
            4.0  0.008240879836451815  \\
            5.0  0.007345490393495181  \\
            6.0  0.00650301009314599  \\
            7.0  0.005713311054911668  \\
            8.0  0.004976318050307804  \\
            9.0  0.09077196835520933  \\
            10.0  0.031602072635961  \\
            11.0  0.01144926123457108  \\
            12.0  0.004211812659007419  \\
            13.0  0.0015586746900467803  \\
            14.0  0.0005715497158904058  \\
            15.0  0.00020221257593735953  \\
            16.0  6.6600122620696e-5  \\
            17.0  1.4133587413064274e-5  \\
            18.0  7.673007721205915e-7  \\
            19.0  2.699934753744335e-9  \\
        }
        ;
    \addlegendentry {$\texttt{HAR-S} (5)$}
    \addplot[color={rgb,1:red,0.0;green,0.5451;blue,0.0}, name path={95}, draw opacity={1.0}, line width={3.5}, solid, mark={*}, mark size={3.0 pt}, mark repeat={1}, mark options={color={rgb,1:red,0.0;green,0.0;blue,0.0}, draw opacity={1.0}, fill={rgb,1:red,0.0;green,0.5451;blue,0.0}, fill opacity={1.0}, line width={0.75}, rotate={0}, solid}]
        table[row sep={\\}]
        {
            \\
            1.0  0.6242779203114784  \\
            2.0  0.6242779203114784  \\
            3.0  0.009189261764723162  \\
            4.0  0.008240879836451815  \\
            5.0  0.007345490393495181  \\
            6.0  0.00650301009314599  \\
            7.0  0.005713311054911668  \\
            8.0  0.004976318050307804  \\
            9.0  0.0042918433379935264  \\
            10.0  0.0036598744543256886  \\
            11.0  0.0030799972974173404  \\
            12.0  0.0025521638779829386  \\
            13.0  0.0020760389547805876  \\
            14.0  0.026634952758776073  \\
            15.0  0.009568211917622079  \\
            16.0  0.0034913621700808847  \\
            17.0  0.0012754076622630173  \\
            18.0  0.00045746227022900877  \\
            19.0  0.00015406287989809907  \\
            20.0  4.260880320699979e-5  \\
            21.0  5.555552872642982e-6  \\
            22.0  1.214219888642261e-7  \\
            23.0  7.071538155330963e-11  \\
        }
        ;
    \addlegendentry {$\texttt{HAR-S} (10)$}
\end{axis}
\end{tikzpicture}

%% file: figs/e-logistic-splice-k.tex
% Recommended preamble:
% \usetikzlibrary{arrows.meta}
% \usetikzlibrary{backgrounds}
% \usepgfplotslibrary{patchplots}
% \usepgfplotslibrary{fillbetween}
% \pgfplotsset{%
%     layers/standard/.define layer set={%
%         background,axis background,axis grid,axis ticks,axis lines,axis tick labels,pre main,main,axis descriptions,axis foreground%
%     }{
%         grid style={/pgfplots/on layer=axis grid},%
%         tick style={/pgfplots/on layer=axis ticks},%
%         axis line style={/pgfplots/on layer=axis lines},%
%         label style={/pgfplots/on layer=axis descriptions},%
%         legend style={/pgfplots/on layer=axis descriptions},%
%         title style={/pgfplots/on layer=axis descriptions},%
%         colorbar style={/pgfplots/on layer=axis descriptions},%
%         ticklabel style={/pgfplots/on layer=axis tick labels},%
%         axis background@ style={/pgfplots/on layer=axis background},%
%         3d box foreground style={/pgfplots/on layer=axis foreground},%
%     },
% }

\begin{tikzpicture}[/tikz/background rectangle/.style={fill={rgb,1:red,1.0;green,1.0;blue,1.0}, fill opacity={1.0}, draw opacity={1.0}}, show background rectangle]
\begin{axis}[point meta max={nan}, point meta min={nan}, legend cell align={left}, legend columns={1}, title={Logistic Regression $\texttt{name}:=\texttt{splice}$, $n:=$60, $N:=$1000}, title style={at={{(0.5,1)}}, anchor={south}, font={{\fontsize{22 pt}{28.6 pt}\selectfont}}, color={rgb,1:red,0.0;green,0.0;blue,0.0}, draw opacity={1.0}, rotate={0.0}, align={center}}, legend style={color={rgb,1:red,0.0;green,0.0;blue,0.0}, draw opacity={1.0}, line width={1}, solid, fill={rgb,1:red,1.0;green,1.0;blue,1.0}, fill opacity={1.0}, text opacity={1.0}, font={{\fontsize{16 pt}{20.8 pt}\selectfont}}, text={rgb,1:red,0.0;green,0.0;blue,0.0}, cells={anchor={center}}, at={(1.02, 1)}, anchor={north west}}, axis background/.style={fill={rgb,1:red,1.0;green,1.0;blue,1.0}, opacity={1.0}}, anchor={north west}, xshift={1.0mm}, yshift={-1.0mm}, width={145.4mm}, height={125.0mm}, scaled x ticks={false}, xlabel={$\textrm{Iterations}$}, x tick style={color={rgb,1:red,0.0;green,0.0;blue,0.0}, opacity={1.0}}, x tick label style={color={rgb,1:red,0.0;green,0.0;blue,0.0}, opacity={1.0}, rotate={0}}, xlabel style={at={(ticklabel cs:0.5)}, anchor=near ticklabel, at={{(ticklabel cs:0.5)}}, anchor={near ticklabel}, font={{\fontsize{16 pt}{20.8 pt}\selectfont}}, color={rgb,1:red,0.0;green,0.0;blue,0.0}, draw opacity={1.0}, rotate={0.0}}, xmajorgrids={true}, xmin={0.27999999999999936}, xmax={25.72}, xticklabels={{$5$,$10$,$15$,$20$,$25$}}, xtick={{5.0,10.0,15.0,20.0,25.0}}, xtick align={inside}, xticklabel style={font={{\fontsize{15 pt}{19.5 pt}\selectfont}}, color={rgb,1:red,0.0;green,0.0;blue,0.0}, draw opacity={1.0}, rotate={0.0}}, x grid style={color={rgb,1:red,0.0;green,0.0;blue,0.0}, draw opacity={0.1}, line width={0.5}, solid}, axis x line*={left}, x axis line style={color={rgb,1:red,0.0;green,0.0;blue,0.0}, draw opacity={1.0}, line width={1}, solid}, scaled y ticks={false}, ylabel={$\|\nabla f\|$}, y tick style={color={rgb,1:red,0.0;green,0.0;blue,0.0}, opacity={1.0}}, y tick label style={color={rgb,1:red,0.0;green,0.0;blue,0.0}, opacity={1.0}, rotate={0}}, ylabel style={at={(ticklabel cs:0.5)}, anchor=near ticklabel, at={{(ticklabel cs:0.5)}}, anchor={near ticklabel}, font={{\fontsize{16 pt}{20.8 pt}\selectfont}}, color={rgb,1:red,0.0;green,0.0;blue,0.0}, draw opacity={1.0}, rotate={0.0}}, ymode={log}, log basis y={10}, ymajorgrids={true}, ymin={4.5084279911009754e-11}, ymax={0.9053260315968684}, yticklabels={{$10^{-10}$,$10^{-8}$,$10^{-6}$,$10^{-4}$,$10^{-2}$,$10^{-1}$}}, ytick={{1.0e-10,1.0e-8,1.0e-6,0.0001,0.01,0.1}}, ytick align={inside}, yticklabel style={font={{\fontsize{15 pt}{19.5 pt}\selectfont}}, color={rgb,1:red,0.0;green,0.0;blue,0.0}, draw opacity={1.0}, rotate={0.0}}, y grid style={color={rgb,1:red,0.0;green,0.0;blue,0.0}, draw opacity={0.1}, line width={0.5}, solid}, axis y line*={left}, y axis line style={color={rgb,1:red,0.0;green,0.0;blue,0.0}, draw opacity={1.0}, line width={1}, solid}, colorbar={false}]
    \addplot[color={rgb,1:red,0.0;green,0.0;blue,0.0}, name path={136}, draw opacity={1.0}, line width={3.5}, solid, mark={*}, mark size={3.0 pt}, mark repeat={1}, mark options={color={rgb,1:red,0.0;green,0.0;blue,0.0}, draw opacity={1.0}, fill={rgb,1:red,0.0;green,0.0;blue,0.0}, fill opacity={1.0}, line width={0.75}, rotate={0}, solid}]
        table[row sep={\\}]
        {
            \\
            1.0  0.4626120049008538  \\
            2.0  0.4567080672454876  \\
            3.0  0.4517324746306242  \\
            4.0  0.4480434413205981  \\
            5.0  0.36582751875468433  \\
            6.0  0.2175705722280067  \\
            7.0  0.05944037958181193  \\
            8.0  0.03199609621556658  \\
            9.0  0.02208002516788396  \\
            10.0  0.017218464300385185  \\
            11.0  0.011631701305345209  \\
            12.0  0.010745609288506569  \\
            13.0  0.010410992210594536  \\
            14.0  0.008051990651081272  \\
            15.0  0.00643642927343116  \\
            16.0  0.004834851885138343  \\
            17.0  0.004688921875707661  \\
            18.0  0.00454982619720695  \\
            19.0  0.003255769414136485  \\
            20.0  0.001930651843353432  \\
            21.0  0.0015643092434216002  \\
            22.0  0.00021203945817145894  \\
            23.0  3.363901891143607e-6  \\
            24.0  5.180783726584079e-9  \\
        }
        ;
    \addlegendentry {$\texttt{ArC}$}
    \addplot[color={rgb,1:red,0.5294;green,0.8078;blue,1.0}, name path={137}, draw opacity={1.0}, line width={3.5}, solid, mark={*}, mark size={3.0 pt}, mark repeat={1}, mark options={color={rgb,1:red,0.0;green,0.0;blue,0.0}, draw opacity={1.0}, fill={rgb,1:red,0.5294;green,0.8078;blue,1.0}, fill opacity={1.0}, line width={0.75}, rotate={0}, solid}]
        table[row sep={\\}]
        {
            \\
            1.0  0.4626120049008538  \\
            2.0  0.4626120049008538  \\
            3.0  0.4626120049008538  \\
            4.0  0.4626120049008538  \\
            5.0  0.21493321178589186  \\
            6.0  0.05132894702593824  \\
            7.0  0.03200659467755593  \\
            8.0  0.025635808748274813  \\
            9.0  0.025587713470211115  \\
            10.0  0.014251015426799388  \\
            11.0  0.02408328459461975  \\
            12.0  0.030819022519708448  \\
            13.0  0.00539284403037875  \\
            14.0  0.0006252105235999514  \\
            15.0  1.2675787969385716e-5  \\
            16.0  6.421236789111247e-9  \\
        }
        ;
    \addlegendentry {$\texttt{HAR-C} (5)$}
    \addplot[color={rgb,1:red,0.0;green,0.6039;blue,0.8039}, name path={138}, draw opacity={1.0}, line width={3.5}, solid, mark={*}, mark size={3.0 pt}, mark repeat={1}, mark options={color={rgb,1:red,0.0;green,0.0;blue,0.0}, draw opacity={1.0}, fill={rgb,1:red,0.0;green,0.6039;blue,0.8039}, fill opacity={1.0}, line width={0.75}, rotate={0}, solid}]
        table[row sep={\\}]
        {
            \\
            1.0  0.4626120049008538  \\
            2.0  0.4626120049008538  \\
            3.0  0.4626120049008538  \\
            4.0  0.4626120049008538  \\
            5.0  0.21493321178589186  \\
            6.0  0.05132894702593824  \\
            7.0  0.03200659467755593  \\
            8.0  0.031251107359075485  \\
            9.0  0.030489838535348916  \\
            10.0  0.029702543845300307  \\
            11.0  0.02886347857978867  \\
            12.0  0.028003181893880836  \\
            13.0  0.04348070178847007  \\
            14.0  0.03222748378777208  \\
            15.0  0.02239335100209114  \\
            16.0  0.03899949768160043  \\
            17.0  0.0018030271555292  \\
            18.0  0.00012077120769574996  \\
            19.0  1.0555498863936152e-6  \\
            20.0  3.8867224063929203e-10  \\
        }
        ;
    \addlegendentry {$\texttt{HAR-C} (10)$}
    \addplot[color={rgb,1:red,0.5647;green,0.9333;blue,0.5647}, name path={139}, draw opacity={1.0}, line width={3.5}, solid, mark={*}, mark size={3.0 pt}, mark repeat={1}, mark options={color={rgb,1:red,0.0;green,0.0;blue,0.0}, draw opacity={1.0}, fill={rgb,1:red,0.5647;green,0.9333;blue,0.5647}, fill opacity={1.0}, line width={0.75}, rotate={0}, solid}]
        table[row sep={\\}]
        {
            \\
            1.0  0.4626120049008538  \\
            2.0  0.4626120049008538  \\
            3.0  0.4626120049008538  \\
            4.0  0.4626120049008538  \\
            5.0  0.21493321178589186  \\
            6.0  0.05132894702593824  \\
            7.0  0.03200659467755593  \\
            8.0  0.031251107359075485  \\
            9.0  0.030489838535348916  \\
            10.0  0.029702543845300307  \\
            11.0  0.02886347857978867  \\
            12.0  0.023561667924066183  \\
            13.0  0.013242409534662443  \\
            14.0  0.007784036868548081  \\
            15.0  0.00839663595175728  \\
            16.0  0.0034624171225121304  \\
            17.0  0.0003466253810096188  \\
            18.0  2.4223364556176894e-5  \\
            19.0  1.4626134077449292e-7  \\
            20.0  8.82293839909849e-11  \\
        }
        ;
    \addlegendentry {$\texttt{HAR-S} (5)$}
    \addplot[color={rgb,1:red,0.0;green,0.5451;blue,0.0}, name path={140}, draw opacity={1.0}, line width={3.5}, solid, mark={*}, mark size={3.0 pt}, mark repeat={1}, mark options={color={rgb,1:red,0.0;green,0.0;blue,0.0}, draw opacity={1.0}, fill={rgb,1:red,0.0;green,0.5451;blue,0.0}, fill opacity={1.0}, line width={0.75}, rotate={0}, solid}]
        table[row sep={\\}]
        {
            \\
            1.0  0.4626120049008538  \\
            2.0  0.4626120049008538  \\
            3.0  0.4626120049008538  \\
            4.0  0.4626120049008538  \\
            5.0  0.21493321178589186  \\
            6.0  0.05132894702593824  \\
            7.0  0.03200659467755593  \\
            8.0  0.031251107359075485  \\
            9.0  0.030489838535348916  \\
            10.0  0.029702543845300307  \\
            11.0  0.02886347857978867  \\
            12.0  0.028003181893880836  \\
            13.0  0.02694858011949281  \\
            14.0  0.02567553973662527  \\
            15.0  0.024462802726548796  \\
            16.0  0.023459224179204992  \\
            17.0  0.02195242778477703  \\
            18.0  0.01778985578757096  \\
            19.0  0.007121463418262018  \\
            20.0  0.0041796928807124835  \\
            21.0  0.004442860978858781  \\
            22.0  0.0012504586298950318  \\
            23.0  0.00015460514116068705  \\
            24.0  2.672121512656306e-6  \\
            25.0  1.5019180388771762e-9  \\
        }
        ;
    \addlegendentry {$\texttt{HAR-S} (10)$}
\end{axis}
\end{tikzpicture}